\documentclass[12pt]{amsart} 
 
\usepackage{amsfonts,amssymb,amsmath} 
\usepackage{latexsym} 
 
\numberwithin{equation}{section} 
 
 \newtheorem{theorem}{Theorem}[section] 
\newtheorem{proposition}[theorem]{Proposition} 
\newtheorem{corollary}[theorem]{Corollary} 
 
\newtheorem{lemma}[theorem]{Lemma} 
\newtheorem{remark}[theorem]{Remark} 
 
\def\c{{\mathbb C}} 
\def\r{{\mathbb R}} 
 
\begin{document} 
 
\title [Kostant--Kirillov form and nilpotent orbits]{On the exactness of 
Kostant--Kirillov form and the second cohomology of nilpotent orbits} 
 
\author[I. Biswas]{Indranil Biswas} 
 
\address{School of Mathematics, Tata Institute of 
Fundamental Research, Homi Bhabha Road, Mumbai 400005, India} 
 
\email{indranil@math.tifr.res.in} 
 
\author[P. Chatterjee]{Pralay Chatterjee} 
 
\address{The Institute of Mathematical Sciences, C.I.T. Campus, 
Taramani, Chennai 600113, India}

\email{pralay@imsc.res.in} 
 
 
 
\begin{abstract} 
We give a criterion for the Kostant-Kirillov form on an adjoint orbit in a real 
semisimple Lie group to be exact. We explicitly compute the second cohomology 
of all the nilpotent adjoint orbits in every complex simple Lie algebras.
\end{abstract}
 
\maketitle 

\section{Introduction} 
We recall a theorem of Arnold (Theorem 1 in page 100 of \cite{Ar}):

\begin{theorem}[\cite{Ar}]\label{th.ar}
Let $\mathcal O$ be the orbit, for the adjoint action of
${\rm SL}(n, {\mathbb C})$, of a $n\times n$ complex matrix with distinct
eigenvalues, and let $\Omega$ be the Kostant--Kirillov holomorphic
symplectic form on $\mathcal O$. Then the real symplectic form ${\rm Im}\,
\Omega$ is exact if and only if all the eigenvalues of a (hence any)
matrix in $\mathcal O$ are purely imaginary; if ${\rm Im}\,
\Omega$ is exact, then $({\mathcal O},\, {\rm Im}\,
\Omega)$ is sympectomorphic to the total space of the cotangent bundle
of the variety of complete flags in ${\mathbb C}^n$ equipped with the
Liouville symplectic form.
\end{theorem}

This theorem of Arnold inspired the work \cite{A-B-B} where it was generalized
to all semisimple adjoint orbits in a complex semisimple Lie algebra.

The following proposition
gives a criterion for the Kostant--Kirillov form on a general
adjoint orbit to be exact. This criterion (which
for $\text{Lie}({\rm SL}(n, {\mathbb C}))$ coincides with the criterion in 
Theorem \ref{th.ar}) can be regarded as a further generalization
of Arnold's result as it applies to all adjoint orbits, thus removing the  
``semisimplicity'' restriction in \cite{A-B-B}.

\begin{proposition}[Proposition \ref{iff-orbit-2} and Proposition
\ref{complex-iff-orbit-2}]\label{prop-i2}
Let $G$ be a real semisimple Lie group such that the
center of $G$ is a finite group. Further assume that a
maximal compact subgroup of $G$ is semisimple.
Let $X \in {\rm Lie}(G)$ be an arbitrary element, and let ${\mathcal O}_X
\subset {\rm Lie}(G)$ be the orbit of $X$ for the adjoint action
of $G$. Let $\omega\,
\in\, \Omega^2 ({\mathcal O}_X )$ be the Kostant--Kirillov symplectic
form on ${\mathcal O}_X$. Then $\omega$ is exact if and only if
all the eigenvalues of the linear
operator ${\rm ad}(X) \,:\, {\rm Lie}(G)\,
\longrightarrow\, {\rm Lie}(G)$ are real.

Let $G$ be a connected complex semisimple Lie group.
Let $X \, \in\, {\rm Lie}(G)$ be an arbitrary element. Let
$\omega$ be the Kostant--Kirillov holomorphic symplectic form
on the orbit of $X$ for the adjoint action of $G$ on ${\rm Lie}(G)$.
Then the form
${\rm Re} \,\omega$ (respectively, ${\rm Im} \,\omega$) is exact
if and only if all the eigenvalues of the linear
operator ${\rm ad} (X) \,:\, {\rm Lie}(G)
\,\longrightarrow\, {\rm Lie}(G)$ are real (respectively,
purely imaginary).
\end{proposition}

{}From Proposition \ref{prop-i2} it follows immediately that the
Kostant--Kirillov form on a nilpotent orbit of ${\rm Lie}(G)$ is
exact. 

As before, let $G$ be a real semisimple Lie group with finite center. Let
$X \in {\rm Lie}(G)$ such that  all eigenvalues of ${\rm ad}(X) \,:\, {\rm 
Lie}(G)\, \longrightarrow\, {\rm Lie}(G)$ are real. So the second cohomology 
class of the Kostant--Kirillov symplectic form on ${\mathcal O}_X$ is zero by 
the first part of Proposition \ref{prop-i2}. It is now natural to ask for a 
description of the full second cohomology group of such orbits. A very rich 
subclass of such orbits are the nilpotent ones, that is, orbits ${\mathcal O}_X$ 
with the property that all the eigenvalues of ${\rm ad}(X) \,:\, {\rm Lie}(G)\,
\longrightarrow\, {\rm Lie}(G)$ are zero. While various topological aspects of 
such orbits have drawn attention over the years (see, for example, Chapter 6 of
\cite{Co-M} and references therein), explicit computation of de Rham cohomology 
groups are not available in the literature to the best of our knowledge. Towards 
this, we completely determine the second de Rham cohomology of the nilpotent 
adjoint orbits in all the complex simple Lie algebras. (See Theorem \ref{a_n} 
(for ${\mathfrak s}{\mathfrak l}_{n}$), Theorem \ref{c_n} (for ${\mathfrak 
s}{\mathfrak p}_{2n}$), Theorem \ref{bd_n} (for ${\mathfrak s}{\mathfrak o}_n$),
Theorem \ref{gfe} (for ${\mathfrak g}_2$, ${\mathfrak f}_4$, ${\mathfrak e}_7$ 
and ${\mathfrak e}_8$) and Theorem \ref{e_6} (for ${\mathfrak e}_6$).) In  
particular, our computations yield the following vanishing result (this
theorem is the combination of Theorem \ref{c_n} and Theorem \ref{gfe}):

\begin{theorem}
Let $\mathfrak g$ be one of the following complex simple Lie
algebras: ${\mathfrak s}{\mathfrak p}_{2n}$, ${\mathfrak g}_2$,
${\mathfrak f}_4$, ${\mathfrak e}_7$ and ${\mathfrak e}_8$. Let
$G$ be a connected Lie group with Lie algebra $\mathfrak g$. The orbit of a
nilpotent element $X \,\in\, {\mathfrak g}$ for the adjoint action of
$G$ on $\mathfrak g$ will be denoted by ${\mathcal O}_X$. Then
$$
H^2 ({\mathcal O}_X, \, \r)\,=\,0\, .
$$
\end{theorem}

In the proof of the above results we require a suitable formulation 
of the second cohomology group of a large class of homogeneous spaces
and a convenient characterization of the exactness of closed two forms; 
see Theorem \ref{2nd-cohom-r} and Corollary \ref{exactness}. 
Although general theories of cohomology groups of (compact) homogeneous spaces 
are extensively studied in the past 
(see, for example, \cite{B}, \cite{C-E} and \cite{G-H-V}) we are 
unable to locate Theorem \ref{2nd-cohom-r} and Corollary \ref{exactness}
in the literature and hence proofs are included for the sake of completeness.
On the other hand, they may be of independent interest as the above results are 
established as a direct consequence of the well--known Whitehead lemma. 

The paper is organized as follows. In \S~2 we fix some standard 
notations and in \S~3 discusses the basic results on second 
cohomology groups of homogeneous spaces that are needed in proving the
main results. In \S~4 we prove results on 
the exactness of Kostant--Kirillov forms, and the final \S~5 is 
devoted to computation of the second cohomology groups of 
nilpotent orbits.

\section{Notation}\label{sec-notation}

In this section we fix some notation, which will be used throughout
this paper. A few specialized notation
are mentioned as and when they occur. 

The {\it center} of a group $G$ is denoted by $Z(G)$ while the {\it center} of a Lie algebra 
${\mathfrak g}$ is denoted by ${\mathfrak z}({\mathfrak g})$. 
Let now $G$ be a Lie group with Lie algebra $\mathfrak g$. The
connected component of $G$ containing the
identity element is denoted by $G^0$.
For a subgroup $H$ of $G$ and a subset 
$S$ of ${\mathfrak g}$, by $Z_{H} (S)$ we will denote the 
subgroup of $H$ that fixes $S$ pointwise. 
Similarly, for a Lie subalgebra ${\mathfrak h}\, \subset\,
\mathfrak g$, by  $Z_{\mathfrak h} (S)$ we will denote the
subalgebra of $\mathfrak h$ that commute with every element of $S$. 

Let $\Gamma$ be a group acting linearly on a vector space $V$. The subspace of fixed points, under the action of
$\Gamma$, is denoted by $V^\Gamma$. If $G$ is a Lie group with Lie
algebra $\mathfrak g$, then
it is immediate that the adjoint (respectively, the coadjoint)  
action of $G^0$ on ${\mathfrak z}(\mathfrak g)$ (respectively, on 
${\mathfrak z} ({\mathfrak k})^*$) is trivial;
in particular, one obtains a natural 
action of $G/G^0$ on ${\mathfrak z}(\mathfrak g)$ (respectively, on
${\mathfrak z} ({\mathfrak k})^*$). We denote by
$[{\mathfrak z} ({\mathfrak g})]^{G/G^0}$ (respectively,
$[{\mathfrak z} ({\mathfrak g})^*]^{G/G^0}$)
the space of fixed points of ${\mathfrak z} ({\mathfrak g})$ (respectively, of ${\mathfrak z} ({\mathfrak g})^*$)
under the action of $G/G^0$.

\section{The second cohomology of homogeneous spaces} 
 
We begin by recalling some well-known results that will be useful here. 
 
\begin{theorem}\label{whitehead} 
Let $k$ be either $\mathbb R$ or $\mathbb C$. 
Let $\mathfrak g$ be a semisimple Lie
algebra over $k$, and let $\rho$ be a
representation of $\mathfrak g$ in a finite dimensional vector space
$V$ over $k$. Then 
$$ 
H^1 ({\mathfrak g}, V) =0 ~\,~\,~ and ~\,~\,~ H^2 ({\mathfrak g}, V)
=0\, . 
$$ 
\end{theorem} 
 
Theorem \ref{whitehead} is known as the Whitehead lemma; see
Theorem 3.12 in page 220 of \cite{V} for a proof.
 
The following is a minor variation of
a result of G. D. Mostow, (Theorem 3.1 in page 260 of \cite{M}),
on the homotopy type of a homogeneous space. 
 
\begin{theorem}\label{mostow} 
Let $G$ be a connected Lie group, and let $H\, \subset\, G$ be a closed
subgroup with finitely many 
connected components. Let $M$ be a maximal compact subgroup of $G$ such
that $M \bigcap H$ is a maximal compact subgroup 
of $H$. Then the image of the natural embedding 
$M/ (M \bigcap H)\, \hookrightarrow\, G/H$ is a deformation 
retraction of $G/H$. 
\end{theorem} 
 
In Theorem 3.1 (in page 260) of \cite{M} it is assumed that $H$ is connected.
But, in view of Theorem 3.1 (in page 180) of \cite{H}, 
the proof of Theorem 3.1 in \cite{M} goes through in the 
more general situation 
where $H$ is not connected but the number of connected components
is finite. 
 
\begin{theorem}\label{2nd-cohom-r} 
Let $G$ and $H$ be as in Theorem \ref{mostow}.
Assume that a (hence any) maximal compact subgroup of $G$ is semisimple.
Let $K$ be a maximal compact subgroup of $H$. 
Let ${\mathfrak k}$ be the Lie algebra of $K$. 
Then there is a natural isomorphism between the real vector spaces
$[{\mathfrak z} 
({\mathfrak k})^*]^{K/K^0}$ and $H^2(G/H, \, \r)$, where $K_0\, \subset\, K$
is the connected component containing the identity element. 
\end{theorem} 
 
\begin{proof} 
Let $M$ be a maximal compact subgroup of $G$ containing 
$K$. Then by Theorem \ref{mostow}, the image of the 
quotient space $M/ K$ under the natural embedding 
$M/ K\, \hookrightarrow\, G/H$ is a deformation 
retraction of $G/H$. In particular,
\begin{equation}\label{f0} 
H^i(G/H,\, \r) \,=\, H^i (M/K, \,\r) 
\end{equation} 
for all $i$. 
 
Since $G$ is connected, its maximal compact subgroup $M$ 
is also connected. Let ${\mathfrak m}$ denote the Lie algebra of $M$. 
We need to recall some standard notation. 
 
For $n\, \geq\, 0$, 
let $\Omega^n ( {\mathfrak m})$ denote the space of all
$n$-forms on the Lie algebra ${\mathfrak m}$. 
Let 
$$ 
d \,:\, \Omega^n ( {\mathfrak m}) \,\longrightarrow\, 
\Omega^{n+1} ({\mathfrak m}) 
$$ 
be defined by 
$$ 
w (X_1, \cdots, X_{n+1}) \, \longmapsto\, \sum_{i < j} (-1)^{i+j} 
w ([X_i, X_j], X_1, \cdots, \widehat{X_i}, \cdots , \widehat{X_j}, 
\cdots, X_{n+1})\, ,
$$ 
where $X_1, \cdots , X_{n+1} \in {\mathfrak m}$. An $n$-form $w \,\in\,
\Omega^n ( {\mathfrak m})$ is said to {\it annihilate} ${\mathfrak k}$ if 
$$w (X_1, \cdots, X_n)\,=\,0$$ whenever $X_i \in {\mathfrak k}$ for 
some $i$. The group $M$ has the adjoint action on the Lie
algebra ${\mathfrak m}\,=\, {\rm Lie}(M)$ which in turn induces an action of $M$ on
each $\Omega^n ( {\mathfrak m})$; this action of $M$ 
on $\Omega^n ( {\mathfrak m})$ has the following description: 
$$ 
({\rm Ad} (g)^* w) 
(X_1, \cdots, X_n) \,=\, ({\rm Ad}(g)X_1, \cdots, {\rm Ad}(g)X_n) 
$$ 
for all $w \in \Omega^n ( {\mathfrak m})$ and for all $X_1, \cdots ,X_n 
\in {\mathfrak m}$. Let 
$$\Omega^n ( {\mathfrak m}/ {\mathfrak k})^K\, \subset\, 
\Omega^n ( {\mathfrak m}) 
$$ 
be the subspace of $n$-forms that simultaneously
annihilate ${\mathfrak k}$ and ${\rm Ad}(g)^* (w) \,=\, w$ for all $g \in K$.
It is easy to see that $$d (\Omega^n ( {\mathfrak m}/ {\mathfrak k})^K) 
\,\subset \, 
\Omega^{n+1} ( {\mathfrak m}/ {\mathfrak k})^K\, ,$$ 
and consequently the pair $\{\Omega^* ( {\mathfrak m}/ {\mathfrak k})^K 
, d \} $ is a sub-complex of the complex 
$\{\Omega^* ( {\mathfrak m}), d \}$. 
 
As $M$ is compact and connected, from Theorem 30 (in page 310) of \cite{S} 
and the formula given in page 313 of \cite{S}, we conclude that there 
are natural isomorphisms
\begin{equation}\label{e1}
H^i ( G/H, \r)\,=\,
H^i ( M/K, \r) \, = \, \frac{ {\rm Ker} \,( d:\Omega^i ( 
{\mathfrak m}/ {\mathfrak k})^K 
\to \Omega^{i+1} ( {\mathfrak m}/ {\mathfrak k})^K)}{ d (\Omega^{i-1} ( {\mathfrak m}/ {\mathfrak k})^K)} 
\end{equation}
for all $i$; the first isomorphism is the one in \eqref{f0}.
 
Setting $i\,=\, 2$ in \eqref{e1}, we have
\begin{equation}\label{e2} 
H^2 (G/H, \r) \,=\, \frac{ {\rm Ker} \,( d:\Omega^2 ( {\mathfrak m}/ 
{\mathfrak k})^K 
\to \Omega^3 ( {\mathfrak m}/ {\mathfrak k})^K)}{ d (\Omega^1 ( 
{\mathfrak m}/ {\mathfrak k})^K)}\, . 
\end{equation} 
We will identify the numerator and the denominator in \eqref{e2}. 
 
Since $M$ is semisimple, its Lie algebra ${\mathfrak m}$ is also semisimple. 
Consequently, by Theorem \ref{whitehead}, we have $$H^2 ( {\mathfrak m}, \r) 
\,=\, 0\, ,$$ where 
$\r$ is the trivial representation of ${\mathfrak m}$. Thus, we have
\begin{equation}\label{e4} 
{\rm Ker} \,( d:\Omega^2 ( {\mathfrak m}/ {\mathfrak k})^K 
\to \Omega^3 ( {\mathfrak m}/ {\mathfrak k})^K)\, = \, 
\Omega^2 ( {\mathfrak m}/ {\mathfrak k})^K \cap d (\Omega^1 
( {\mathfrak m}))\, . 
\end{equation} 
 
The map
\begin{equation}\label{i-m}
d \,:\, \Omega^1 ( {\mathfrak m}) \,\longrightarrow\, 
\Omega^2 ( {\mathfrak m})
\end{equation}
is injective, because 
${\mathfrak m}$ is semisimple. 
Further, for any $g\, \in\,K$, the operators ${\rm Ad} (g)^*$ and $d$ 
on the space of forms commute. 
These two facts together imply that a form $\eta \,\in\, \Omega^1 ( 
{\mathfrak m})$ 
is invariant under the adjoint action of $K$ if and only if $d 
\eta$ is invariant under the adjoint action of $K$. Therefore, 
\begin{equation}\label{E2} 
\Omega^2 ( {\mathfrak m}/ {\mathfrak k})^K \cap d (\Omega^1 ( {\mathfrak m})) 
\,=\, \Omega^2 ( {\mathfrak m}/ {\mathfrak k})^K \cap d (\Omega^1 ( 
{\mathfrak m})^K)\, . 
\end{equation} 

We next claim that
\begin{equation}\label{e3}
d (\Omega^1 ( {\mathfrak m})^K)\, \subset\,
\Omega^2 ( {\mathfrak m}/ {\mathfrak k})^K\, .
\end{equation}
To prove this claim, first observe that if
$\eta \in \Omega^1 ( {\mathfrak m})$,
then $\eta$ is invariant under the adjoint action of $K^0$
(the connected component of $K$) if and only
if $\eta \circ {\rm ad}(X) \,=\, 0$
for all $X \in {\mathfrak k}$. This implies that $ d (\Omega^1 (
{\mathfrak m})^K) \subset
\Omega^2 ( {\mathfrak m}/ {\mathfrak k})$. Further, $d (\Omega^1 (
{\mathfrak m})^K) \subset
\Omega^2 ( {\mathfrak m})^K$. Combining these two,
the claim in \eqref{e3} follows immediately.

Combining \eqref{E2} and \eqref{e3} with \eqref{e2}, we have
$$ 
H^2 (G/H, \,\r)\,=\,\frac{ {\rm Ker} \,( d:\Omega^2 ( {\mathfrak m}/ 
{\mathfrak k})^K 
\to \Omega^3 ( {\mathfrak m}/ {\mathfrak k})^K)}{ d (\Omega^1 ( {\mathfrak m}/ {\mathfrak k})^K)} 
\,=\, \frac{d (\Omega^1 ( {\mathfrak m})^K)}{d (\Omega^1 ( {\mathfrak m}/ 
{\mathfrak k})^K)} \, . 
$$ 
 
Now the injectivity of $d \,:\, \Omega^1 ( {\mathfrak m})\, 
\longrightarrow\, \Omega^2 ( {\mathfrak m})$ implies that
\begin{equation}\label{e7} 
H^2(G/H, \, \r) \,=\, \frac{d (\Omega^1 ( {\mathfrak 
m})^K)}{d (\Omega^1 ( {\mathfrak m}/ {\mathfrak k})^K)} \,=\, 
\frac{\Omega^1 ( {\mathfrak m})^K}{\Omega^1 ( {\mathfrak m}/ {\mathfrak 
k})^K}\, . 
\end{equation} 
 
Since $K$ is compact, we get an ${\rm Ad}(K)$--invariant complement of 
the subspace ${\mathfrak k}\, \subset\, {\mathfrak m}$; this complement 
will be denoted by 
${\mathfrak k}^{\perp}$. 
Since ${\mathfrak k}$ is reductive (because $K$ is compact), 
it has the direct sum decomposition 
\begin{equation}\label{z}
{\mathfrak k} \,=\, {\mathfrak z} ({\mathfrak k}) + [{\mathfrak k}, 
{\mathfrak k}]\, ; 
\end{equation}
note that $[{\mathfrak k}\, , {\mathfrak k}]$ 
is a semisimple subalgebra of ${\mathfrak k}$. Thus we have a 
direct sum decomposition of the $K$--module 
${\mathfrak m}$ 
\begin{equation}\label{icd}
{\mathfrak m} \,= \, {\mathfrak z} ({\mathfrak k}) 
\oplus [{\mathfrak k}, {\mathfrak k}] 
\oplus {\mathfrak k}^{\perp}\, .
\end{equation}

Let 
\begin{equation}\label{ic-si}
\sigma \,:\, \Omega^1 ( {\mathfrak m})\,=\, {\mathfrak m}^* \, 
\longrightarrow\, {\mathfrak z} 
({\mathfrak k})^* \oplus 
[{\mathfrak k}, {\mathfrak k}]^* \oplus ({\mathfrak k}^{\perp})^*
\end{equation}
be the isomorphism defined by 
$f\, \longmapsto\, ( f\vert_{{\mathfrak z} ({\mathfrak k})}, f\vert_{[ 
{\mathfrak k}, {\mathfrak k}]}, f\vert_{{\mathfrak k}^{\perp}})$. 
As each of the subspaces of ${\mathfrak m}$ in \eqref{icd} is 
${\rm Ad}(K)$--invariant, the isomorphism $\sigma$ in
\eqref{ic-si} induces an isomorphism 
\begin{equation}\label{e6} 
\Omega^1 ( {\mathfrak m})^K \,\stackrel{\sim}{\longrightarrow}\, 
({\mathfrak z} ({\mathfrak k})^*)^K \oplus 
([{\mathfrak k}, {\mathfrak k}]^*)^K \oplus (({\mathfrak k}^{\perp})^*)^K 
\, . 
\end{equation} 
 
We claim that 
\begin{equation}\label{e5} 
([{\mathfrak k}, {\mathfrak k}]^*)^K \,= \,0\, . 
\end{equation} 

To prove \eqref{e5}, take any $\mu \,\in\, 
([{\mathfrak k}, {\mathfrak k}]^*)^K$. Then $\mu \circ {\rm Ad (g)} (X) 
\,= \,\mu (X)$ 
for all $X \,\in\, [{\mathfrak k}, {\mathfrak k}]$ and $g \,\in\, K$. 
By differentiating, one has that $\mu ({\rm ad}(Y)(X)) \,= \,0$ for 
all $X \,\in \,[{\mathfrak k}, {\mathfrak k}]$ and $Y \in 
{\mathfrak k}$. 
Thus $\mu ([{\mathfrak k},[{\mathfrak k}, {\mathfrak k}]]) \,=\, 0$. 
But, as $[{\mathfrak k}, {\mathfrak k}]$ is semisimple, we have 
$$[{\mathfrak k},[{\mathfrak k}, {\mathfrak k}]] \,=\, 
[{\mathfrak z}({\mathfrak k})+ [{\mathfrak k}, {\mathfrak k}] ,[ 
{\mathfrak k}, {\mathfrak k}]] 
\,= \, [ [{\mathfrak k}, {\mathfrak k}] ,[{\mathfrak k}, {\mathfrak 
k}]] = [{\mathfrak k}, {\mathfrak k}]\, .$$ 
Therefore, $\mu ([{\mathfrak k}, {\mathfrak k}]) \,=\, 0$. This proves 
the claim in \eqref{e5}. 
 
Combining \eqref{e6} and \eqref{e5}, the homomorphism 
$$ 
\widetilde{\sigma} \,:\, \Omega^1 ( {\mathfrak m})^K \,\longrightarrow 
\, ({\mathfrak z} ({\mathfrak k})^*)^K\oplus (({\mathfrak k}^{\perp})^*)^K 
$$ 
defined by $f\, \longmapsto\, 
( f\vert_{{\mathfrak z} ({\mathfrak k})}, f\vert_{{\mathfrak 
k}^{\perp}})$ is an isomorphism. 
 
Note that $f ( {\mathfrak k}) \,= \,0$ for every 
$f\, \in \,\Omega^1 ({\mathfrak m}/ {\mathfrak k})^K$; in 
particular, $f ({\mathfrak z} ({\mathfrak k})) \,=\, 0$. Thus 
$\widetilde{\sigma} ( \Omega^1 ( {\mathfrak m}/ {\mathfrak k})^K) = 
\{0\} \oplus (({\mathfrak k}^{\perp})^*)^K$. 
This in turn implies that 
$$ 
\frac{\Omega^1 ( {\mathfrak m})^K}{\Omega^1 ( {\mathfrak m}/ {\mathfrak k})^K} 
\,=\, \frac {({\mathfrak z} ({\mathfrak k})^*)^K 
\oplus (({\mathfrak k}^{\perp})^*)^K}{ \{0\} \oplus (({\mathfrak k}^{\perp})^*)^K} 
\,=\, ({\mathfrak z} ({\mathfrak k})^*)^K \, . 
$$ 
Hence from \eqref{e7}, we have
\begin{equation}\label{cii}
H^2(G/H,\, \r) \, =\, ({\mathfrak z} ({\mathfrak 
k})^*)^K\, . 
\end{equation}
 
As the adjoint action of $K^0$ on ${\mathfrak z} ({\mathfrak k})$ is 
trivial, it follows that $({\mathfrak z} ({\mathfrak k})^*)^K \,=\, 
({\mathfrak z} ({\mathfrak k})^*)^{K/K^0}$. In view of \eqref{cii},
this completes the proof of the theorem. 
\end{proof} 
 
\begin{corollary}\label{cor-1} 
The set--up is same as in Theorem \ref{2nd-cohom-r}. 
\begin{enumerate} 
\item If $K^0$ is semisimple, then $ H^2 (G/H, \,\r) \,= \,0$. 
 
\item If $H$ is connected, then $H^2
(G/H,
\, \r) \,=\,0$ if and only if $K$ is semisimple. 
\end{enumerate} 
\end{corollary} 
 
\begin{proof}
If $K^0$ is semisimple, then ${\mathfrak z} ({\mathfrak 
k})\,=\, 0$. Therefore, the first statement follows readily from 
Theorem \ref{2nd-cohom-r}. 
 
To prove the second statement, assume that $H$ is connected. 
This implies that $K$ is connected. 
Thus, we have $({\mathfrak z} ({\mathfrak 
k})^*)^{K/K^0} = {\mathfrak z} ({\mathfrak k})^*$. Hence from 
Theorem \ref{2nd-cohom-r} it follows that 
$$ 
H^2 (G/H, \, \r) \,=\, 0 
$$ 
if and only if ${\mathfrak z} ({\mathfrak k})\, =\, 0$. The center 
${\mathfrak z} ({\mathfrak k})$ vanishes if and only if 
$K$ is semisimple. 
\end{proof} 

The following corollary of Theorem \ref{2nd-cohom-r} shows that
$H^2(G/H, \, \r)$ is independent of $G$; it depends only on $H$.
 
\begin{corollary}\label{cor1} 
Let $G_1$ and $G_2$ be two connected Lie groups admitting semisimple 
maximal compact subgroups, and let $H$ be a Lie group with finitely many 
connected components. Assume that $H$ is 
embedded in both $G_1$ and in $G_2$ as closed subgroup. Then 
$$ 
H^2(G_1/H, \, \r) \, = \, H^2(G_2/H, \, \r)\, . 
$$ 
\end{corollary} 
 
If the dimension of ${\mathfrak z} ({\mathfrak k})$ is small, and the 
finite group $K/K^0$ is of specific type, then more explicit 
results hold. 
 
\begin{corollary}\label{small-center} 
Let $G$, $H$, $K$ and ${\mathfrak z} 
({\mathfrak k})$ be as in Theorem \ref{2nd-cohom-r}. 
\begin{enumerate} 
\item Assume that $\dim {\mathfrak z} ({\mathfrak k}) \,\leq\, 1$, 
and also assume that ${\mathbb Z}/2{\mathbb Z}$ is not a quotient of
$K/K^0$. Then $[{\mathfrak z} ({\mathfrak 
k})^*]^{K/K^0} \,=\, 
{\mathfrak z} ({\mathfrak k})^*$, and consequently 
$H^2(G/H, \,\r) \,= \,{\mathfrak z} ({\mathfrak k})^*$. 
 
\item Assume that $\dim {\mathfrak z} ({\mathfrak k}) \,\leq\, 2$,
and also assume that $K/K^0$ admits no nontrivial abelian quotient. 
Then $[{\mathfrak z} ({\mathfrak k})^*]^{K/K^0} \,= \,{\mathfrak z} 
({\mathfrak k})^*$, and consequently 
$H^2 (G/H, \,\r) \,= \,{\mathfrak z} ({\mathfrak k})^*$. 
\end{enumerate} 
\end{corollary} 
 
\begin{proof} 
If $\dim {\mathfrak z} ({\mathfrak k})\,=\, 1$, then
${\mathbb Z}/2\mathbb Z$ is the 
only nontrivial finite subgroup of ${\rm Aut}_{\r} ({\mathfrak z} 
({\mathfrak k}))$. Therefore, the first part
of the corollary follows from Theorem \ref{2nd-cohom-r}. 
 
If $\dim {\mathfrak z} ({\mathfrak k})\,=\,2$, then 
any nontrivial finite subgroup of ${\rm Aut}_{\r} ({\mathfrak z} 
({\mathfrak k}))$ has a nontrivial abelian quotient. 
Therefore, the second part of the corollary 
also follows from Theorem \ref{2nd-cohom-r}. 
\end{proof} 
 
We use the set--up of Theorem \ref{2nd-cohom-r} and the notation
established in the proof of Theorem \ref{2nd-cohom-r}. 
Our aim is to get a method, stemming from the proof of 
Theorem \ref{2nd-cohom-r}, which tells us when a closed two--form on 
$G/H$ is exact. Let 
$$\phi \, : M/K \, \longrightarrow \, G/H$$ be the natural embedding. 
Let $w \in \Omega^2 (G/H)$ be such that $dw =0$.
As $\phi (M/K)$ is a deformation retract of $G/H$, 
the closed two--form $$w_2 \,=\, \phi^* w \,\in \,\Omega^2 (M/K)$$ is 
exact if and only if $w$ is exact. 
Now, as in Proposition 28 (in page 309) of \cite{S}, we find a 
$M$--invariant form 
$w_3$ on $M/K$ such that $w_2 - w_3$ is exact. 
Let $$\pi \, :\, M \,\longrightarrow \,M/K$$ be the quotient map. 
We identify the Lie algebra ${\mathfrak m}$ with the tangent space 
$T_eM$ of $M$ at the identity element. 
Consider the ${\rm Ad} (K)$--invariant closed two--form $$\widetilde{w} 
\,\in\, \Omega^2 ({\mathfrak m}/ {\mathfrak k})$$ 
defined by $(X, Y )\, \longmapsto\, w_3 ( eK) ( d \pi_e X, d \pi_e 
Y)$, where $X, Y \in {\mathfrak m}$. 
There is a unique ${\rm Ad} (K)$--invariant 
form $\eta \,\in\, \Omega^1 ({\mathfrak m})$ such that 
$d\eta\, =\, \widetilde{w}$; the existence of $\eta$
follows from \eqref{e4}, and its uniqueness follows from
the injectivity of the map in \eqref{i-m}.

\begin{corollary}\label{exactness} 
With the above notation, the $2$-form $w$ is exact 
if and only if $\eta ( {\mathfrak z} ({\mathfrak k}))\,=\,0$.
\end{corollary}

\begin{proof}
If $w$ is exact, then it was observed in the proof of
Theorem \ref{2nd-cohom-r} that $\eta\, \in\, \Omega^1 (
{\mathfrak m}/ {\mathfrak k})^K$ (see \eqref{e7}). Hence
in that case $\eta ({\mathfrak z} ({\mathfrak k}))\,=\,0$.

Conversely, if $\eta ( {\mathfrak z} ({\mathfrak k}))\,=\,0$,
then from \eqref{z} it follows that $\eta\, \in\, \Omega^1 (
{\mathfrak m}/ {\mathfrak k})^K$. Hence $w$ is exact if
$\eta ( {\mathfrak z} ({\mathfrak k}))\,=\,0$.
\end{proof}

We note 
that $\eta ( {\mathfrak z} ({\mathfrak k}))\,=\,0$ if and only if $\eta 
( {\mathfrak k})\,=\,0$, because ${\mathfrak k}\,=\, {\mathfrak z} 
({\mathfrak k})\oplus [{\mathfrak k}\, ,{\mathfrak k}]$ and $\eta$ is
${\rm Ad} (K)$--invariant.
 
\section{Criterion for exactness of the Kostant--Kirillov form} 

We shall treat the complex case and the real case separately.
 
\subsection{The real case} 

Let ${\mathfrak g}$ be a real semisimple Lie algebra. A semisimple element $A \in {\mathfrak g}$ is
called {\it hyperbolic} if all the eigenvalues of ${\rm ad} (A) : {\mathfrak g} \to {\mathfrak g}$ are
real. On the other hand, a semisimple element $B \in {\mathfrak g}$ is
called {\it compact} if all the eigenvalues of ${\rm ad} (B) : {\mathfrak g} \to {\mathfrak g}$ are purely imaginary.
The elements of ${\mathfrak g}$ has a {\it complete Jordan 
decomposition} (see page 430 of \cite{He}). This means that any element 
$X \,\in\, {\mathfrak g}$ 
can be written uniquely as $X\,=\, X_s + X_n$, where $X_s$ is
semisimple and $X_n$ is nilpotent with $[X_s\, ,X_n]\,=\, 0$. 
The semisimple part $X_s$ can be further written as $X_s = X_k + X_h$, 
where $X_h$ is the hyperbolic part of $X_s$, and $X_k$ is the 
compact part of $X_s$. Thus
\begin{equation}\label{e8} 
X \,= \, X_n + X_k + X_h\, . 
\end{equation} 
The four components $X_s$, $X_n$, $X_k$ and $X_h$ commute with each 
other, and the above decomposition of $X$ is unique. 
 
\begin{proposition}\label{iff-orbit-2} 
Let $G$ be a real semisimple Lie group with Lie algebra $\mathfrak g$ such
that the center of $G$ is a finite group. Further assume that a (hence any) 
maximal compact subgroup of $G$ is semisimple. 
Let $X \in {\mathfrak g}$ be an arbitrary element, and let ${\mathcal O}_X 
\subset {\mathfrak g}$ be the orbit of $X$ for the adjoint action
of $G$. Let $\omega\, 
\in\, \Omega^2 ({\mathcal O}_X )$ be the Kostant--Kirillov symplectic 
form on ${\mathcal O}_X$. Then $\omega$ is exact if and only if 
$X_k \,=\,0$, where $X_k$ is the compact part in \eqref{e8}
(this condition is equivalent to the condition that all the eigenvalues 
of the linear operator ${\rm ad}(X) \,:\, {\mathfrak g}\, 
\longrightarrow\, {\mathfrak g}$ are real).
\end{proposition} 
 
\begin{proof} 
Let $B$ be the Killing form on ${\mathfrak g}$; it is nondegenerate 
because $\mathfrak g$ is semisimple. We first need to make a few 
observations regarding the Killing form $B$. 

\smallskip
\noindent 
{\it Claim 1:}\, Let $N \,\in\, {\mathfrak g}$ be a nilpotent element, 
and let $S\,\in\, {\mathfrak g}$ be a semisimple 
element, such that $[S\, , N] \,=\, 0$. Then $B (S, N) \,=\, 0$.
\smallskip

To prove the above claim, note that ${\rm ad}(N)$ and ${\rm ad}(N)$ commute
because $[S\, , N] \,=\, 0$. Since ${\rm ad}(N)$ is nilpotent, this
implies that ${\rm ad}(S){\rm ad}(N)$ is nilpotent. Consequently,
$B (S, N) \,=\, 0$, proving Claim 1.

\smallskip
\noindent 
{\it Claim 2:}\, Let $E \,\in\, {\mathfrak g}$ be a hyperbolic element, 
and let $R \,\in \, {\mathfrak g}$ 
be a compact element with $[E\, , R] \,=\, 0$. Then $B (E\, , R) =0$. 
\smallskip

To prove this claim, 
consider $C = E + R$. As $[E\, , R]\,=\, 0$, and both $E$ and $R$ are 
semisimple, it follows that $C$ is also semisimple. 
Further, by the 
uniqueness of complete Jordan decomposition one has that $C_h \,=\, E$ 
and $C_k \,=\, R$, where $C_h$ and $C_k$ respectively are the hyperbolic and
compact components of $C$ (see \eqref{e8}). 
 
Since $C$ is semisimple, there is a Cartan subalgebra ${\mathfrak 
t}^\prime$ of ${\mathfrak g}$ such that $C\,\in\,{\mathfrak t}^\prime$. 
Fix a maximal compact subgroup $K$ of $G$. The Lie algebra of $K$ will
be denoted by $\mathfrak k$. Let 
$$\theta \,:\, {\mathfrak g} \,\longrightarrow\, {\mathfrak g}$$ 
be a Cartan involution of
${\mathfrak g}$ such that ${\mathfrak k}\, =\, 
\{X \in {\mathfrak g} \, \mid\, \theta (X) = X\}$. Define 
$$ 
{\mathfrak p} \,:= \, \{ X \in {\mathfrak g} \, \mid \, \theta (X) = 
-X\}\, . 
$$ 
Then we have the Cartan decomposition 
${\mathfrak g} \,=\, {\mathfrak k} \oplus {\mathfrak p}$. 
Using Proposition 6.59 (in page 386) of \cite{K} we get hold of a 
$\theta$--stable Cartan
subalgebra, say ${\mathfrak t}$, such that 
${\mathfrak t}$ and ${\mathfrak t}^\prime$ are $G$--conjugate, 
where ${\mathfrak t}^\prime$ is the above Cartan subalgebra 
containing $C$; in other words, there is an element $g \,\in\, G$ such 
that ${\rm Ad}(g) ({\mathfrak t}^\prime)\, = \,{\mathfrak t}$. 
As ${\mathfrak t}$ is $\theta$--stable, we have
$$
{\mathfrak t} \,=\, ({\mathfrak t} \cap {\mathfrak k}) \oplus 
({\mathfrak t} \cap {\mathfrak p})\, . 
$$

Consider $D \,:=\, {\rm Ad} (g) (C)$. We have $D \,=\, D_1 + D_2$, where 
$D_1 \,\in\, {\mathfrak t} \bigcap {\mathfrak k}$ and 
$D_2 \,\in\, {\mathfrak t} \bigcap {\mathfrak p}$. Thus $D_1$ and $D_2$ 
are compact and hyperbolic semisimple elements respectively, 
and, moreover, $[D_1\, , D_2]\,=\, 0$
as $D_1, D_2 \,\in\, {\mathfrak t}$. Hence by the uniqueness of the 
decomposition in \eqref{e8}, we have
$$D_1 \,= \,D_k\,= \,{\rm Ad} (g) 
(C_k)\, ~\,~\, \text{~and~}\,~\,~\, 
D_2 \,= \,D_h\, = \, {\rm Ad} (g) (C_h)\, ,$$ 
where $D_K$ and $D_h$ respectively
are the compact and hyperbolic parts of $D$. 
As $D_k \, \in \, {\mathfrak k}$ and $D_h \,\in\, {\mathfrak p}$, we have
$$ 
B ( E\, , R) \,=\, B ({\rm Ad} (g) (C_h)\, ,{\rm Ad} (g) (C_k))\,=\, B 
(D_h\, , D_k) \, =\, 0\, . 
$$
This completes the proof of Claim 2. 
 
Now let $\mathbb S$ be the topological 
closure of the one--parameter subgroup 
$$ 
\{ \exp (t X_k) \, \mid\, t \in \r\}\, \subset\, G\, . 
$$ 
This $\mathbb S$ is a (compact) torus of $G$, and it is contained in the 
centralizer $Z_G ( X)$ of $X$ in $G$. As $Z_G (X)$ is a real algebraic 
group, it has finitely many connected components. Therefore, 
there is a maximal compact subgroup $L$ of $Z_G (X)$ containing $\mathbb S$
(Theorem 3.1 in page 180 of \cite{H}).
 
Let ${\mathfrak l}$ and $ {\mathfrak s}$ be the Lie algebras of $L$ and
$\mathbb S$ respectively.
Invoking Corollary \ref{exactness} we see that the 
the Kostant--Kirillov symplectic form $\omega \in \Omega^2 ({\mathcal 
O}_X )$ on the orbit ${\mathcal O}_X$ is exact 
if and only if $B (X\, , {\mathfrak l}) \,= \,0$. 
 
Take any $\xi \,\in \, {\mathfrak l}$. Since $\xi$ is a compact element 
of $Z_{\mathfrak g}(X)$, we have
$$ 
[\xi\, , X_h] \, = \, [\xi\, , X_k] \,= \, [\xi, X_n] \,=\, 0\, .$$ 
Now by Claim 1 and Claim 2, we have
$$ 
B (X\, , \xi)\,=\, B ( X_n\, , \xi) + B (X_h \, , \xi) + B ( X_k \, , \xi) 
\,=\, B(X_k \, , \xi)\, . 
$$
Thus $B (X\, , {\mathfrak l}) \,=\, 0$ if and only if $B (X_k\, , 
{\mathfrak l}) \,= \,0$. 
As the Killing form $B$ is negative definite on any compact subalgebra 
of ${\mathfrak g}$, it is definite 
on ${\mathfrak l}$. Since $X_k \,\in\, {\mathfrak l}$, 
we now conclude the following: 
$B (X_k\, , {\mathfrak l}) \, = \, 0$ if and only if $X_k =0$. 
This completes the proof of the proposition. 
\end{proof} 
 
\begin{remark}\label{re-exact} 
{\rm From Proposition \ref{iff-orbit-2} it follows immediately that 
for any nilpotent element $X$ of a real semisimple Lie algebra 
$\mathfrak g$, the Kostant--Kirillov form on the adjoint orbit 
${\mathcal O}_X$ is exact.} 
\end{remark} 
 
\subsection{The complex case} 
 
Let $G$ be a complex semisimple group with Lie algebra $\mathfrak g$.
Take any element 
$X \,\in\, {\mathfrak g}$. Let ${\mathcal O}_X \,\subset\, {\mathfrak 
g}$ be the orbit of $X$ for the adjoint action of $G$. Let 
$\omega \in \Omega^2 ({\mathcal O}_X )$ be the holomorphic 
Kostant--Kirillov symplectic form on ${\mathcal O}_X$. 
The real and imaginary parts of $\omega$ will be denoted 
by ${\rm Re} \, \omega$ and ${\rm Im} \, \omega$ respectively. 
 
It is observed in Lemma 2.1 of \cite{A-B-B} that if $X$ is semisimple, 
then ${\rm Re} \, \omega$ and ${\rm Im} \, \omega$ are real symplectic 
forms on ${\mathcal O}_X$. However, the proof of Lemma 2.1 of \cite{A-B-B} 
goes through for arbitrary $X$. Therefore, 
${\rm Re} \, \omega$ and ${\rm Im} \, \omega$ are real symplectic 
forms on ${\mathcal O}_X$ for all $X\, \in\, \mathfrak g$. 
 
\begin{proposition}\label{complex-iff-orbit-2} 
Let $X \in {\mathfrak g}$ be an arbitrary element. 
Then ${\rm Re} \,\omega$ (respectively, ${\rm Im} \,\omega$) is exact 
if and only if all the eigenvalues of the linear
operator ${\rm ad} (X) \,:\, {\mathfrak g} 
\,\longrightarrow\, {\mathfrak g}$ are real (respectively, 
purely imaginary). 
\end{proposition} 
 
\begin{proof} 
The group $G$ (respectively, 
the Lie algebra ${\mathfrak g}$) when considered as a real analytic group 
(respectively, real Lie algebra) will be denoted 
by $G^\r$ (respectively, ${\mathfrak g}^\r$). 
As $G$ is complex semisimple, any of the maximal compact subgroups of 
$G^\r$ is a real form of $G$ and thus they are all semisimple. We fix a maximal 
compact subgroup $K$ of $G^\r$, and we denote the Lie algebra of $K$ by 
${\mathfrak k}$. 
 
Let 
\begin{equation}\label{J} 
J \,:\, {\mathfrak g} \,\longrightarrow\, {\mathfrak g} 
\end{equation} 
be the linear map defined by multiplication with $\sqrt{-1}$. 
As ${\mathfrak g}^\r\,= \,{\mathfrak k}\oplus \sqrt{-1}\cdot {\mathfrak 
k}$, we have the conjugation 
$$ 
\theta \,:\, {\mathfrak k}\oplus\sqrt{-1}{\mathfrak k}\,\longrightarrow\, 
{\mathfrak k}\oplus \sqrt{-1}\cdot {\mathfrak k} 
$$ 
defined by $x + \sqrt{-1}\cdot y \, \longmapsto\, x-\sqrt{-1}\cdot y$, 
which is the Cartan involution of ${\mathfrak g}^\r$. The 
Killing form $B$ on the complex Lie algebra ${\mathfrak g}$ 
and the Killing form $B^\r$ on the real Lie algebra 
${\mathfrak g}^\r$ are related by the identity $B^\r \,= \,2 {\rm Re}\, B$. 
 
We will first prove that ${\rm Re} \,\omega$ is exact 
if and only if all the eigenvalues of ${\rm ad}(X)$ are real. 
 
Let $\omega_{\r}$ be the Kostant--Kirillov form on 
${\mathcal O}_X$ when considered as the $G^\r$--orbit of $X$. 
As $B^{\r} = 2 {\rm Re} \, B$, it follows that 
$$ 
\omega_{\r} \,=\, 2{\rm Re} \, \omega\, . 
$$ 
So ${\rm Re} \, \omega$ is exact if and only if $\omega_{\r}$ is 
exact. Now using Theorem \ref{iff-orbit-2} we conclude that 
${\rm Re} \, \omega$ is 
exact if and only if all the eigenvalues of ${\rm ad}(X)$ are real. 
 
To prove the criterion for ${\rm Im} \, \omega$, 
first note that the map $J$ in \eqref{J} induces a 
$G$--equivariant diffeomorphism 
$$ 
\widetilde{J} \,:\, {\mathcal O}_X \,\longrightarrow\, {\mathcal 
O}_{\sqrt{-1}X}\, . 
$$ 
Let $\omega^\prime_{\r}$ be the Kostant--Kirillov form on 
${\mathcal O}_{\sqrt{-1}X}$ constructed by considering it as
the $G^\r$--orbit of $\sqrt{-1}X$. Then one has 
$$ 
\widetilde{J}^* \omega^\prime_{\r} \,=\, {\rm Im} \, \omega\, . 
$$ 
Thus the form ${\rm Im} \, \omega$ is exact if and only if 
$\omega^\prime_{\r}$ is exact. The first 
part of the proposition says that $\omega^\prime_{\r}$ is exact if and 
only if all the eigenvalues of ${\rm ad}(\sqrt{-1}X)$ are real, 
which, in turn, is equivalent to the statement that all the eigenvalues 
of ${\rm ad}(X)$ are purely imaginary. This completes the proof of the 
proposition. 
\end{proof} 

\begin{remark}
{\rm From Proposition \ref{complex-iff-orbit-2} it follows immediately 
that for any nilpotent element $X\, \in \, \mathfrak g$, the two 
$2$--forms ${\rm Re} \,\omega$ and ${\rm Im} \,\omega$ 
on ${\mathcal O}_X$ are exact.}
\end{remark}

\section{Nilpotent orbits in complex semisimple groups}\label{sec-n}

In this section we will compute the second de Rham cohomology of the nilpotent 
orbits in a complex simple group.
But first it will be shown directly that the first and 
second de Rham cohomologies of 
orbits of regular nilpotent elements vanish. Recall that a 
nilpotent element $X$ in a complex simple Lie algebra $\mathfrak g$
is called {\it regular} if $\dim [X\, ,{\mathfrak g}]\, \geq\,
\dim [Y\, ,{\mathfrak g}]$ for all nilpotent $Y\,\in\, \mathfrak g$.

\begin{lemma}\label{regular-nilpotent}
Let $G$ be a complex 
semisimple group. Let $X$ be a regular nilpotent element
in the Lie algebra $\mathfrak g$ of $G$, and let ${\mathcal O}_X\,\subset\,
\mathfrak g$ be the orbit of $X$ under the adjoint action. Then
$$
H^1({\mathcal O}_X,\, {\mathbb R})\,=\,0\,=\, 
H^2({\mathcal O}_X,\, {\mathbb R})\, .
$$
\end{lemma}

\begin{proof}
Let $u \,= \,\exp (X)$.
Let $B$ be a Borel subgroup of $G$ containing 
$u$. The unipotent radical of $B$ will be denoted by $U$. Fix a 
maximal torus $T$ of $B$. As $X$ is regular nilpotent, we have
$$Z_G (X) \, =\, Z(G) \times Z_U(X)\, .$$ Let $A$ be the 
connected maximal $\r$--split part of $T$.

By Iwasawa decomposition, 
there is a maximal compact subgroup $K$ of $G$ such that
the natural map $K \times (AU)\, \longrightarrow\, G$ is a 
diffeomorphism. Consequently, the natural map
$$
(K/ Z(G)) \times (A(U/Z_U (X)))\, \longrightarrow\, G/Z_G (X)
$$
is a diffeomorphism.
As $(AU)/Z_G (X) \,=\, A \times (U/Z_U (X))$ is diffeomorphism 
to $\r^n$, we have
\begin{equation}\label{f1}
H^i (G/Z_G(X), \,\r)\,=\, H^i ( K/ Z(G),\, \r)
\end{equation}
for all $i$. Since $K$ is a real form of $G$, we know that $K$ is compact, 
connected and semisimple. Hence $H^1 ( K/ Z(G),\, \r) \, =\, 0
\,=\, H^2(K/Z(G),\, \r)$ by Theorem \ref{whitehead}.
Hence the proof is completed by \eqref{f1}.
\end{proof}

Let $G$ be a complex semisimple Lie group with Lie algebra
$\mathfrak g$. If $X$ is a nonzero nilpotent element,
then Jacobson-Morozov theorem says that there exist $E, Y \,\in \, {\mathfrak g}$
such that $[E,X] \,=\, 2X$, $[E,Y]\,=\,-2Y$ and $[X, Y]\,=\,E$. Clearly the subalgebra generated
by $X, E, Y$ is isomorphic to ${\mathfrak{s} \mathfrak{l}}_2$;
the triple $(X,E,Y)$ is called a {\it ${\mathfrak{s}\mathfrak{l}}_2$-triple} containing
$X$ as a {\it nil-positive} element. (See Theorem 3.3.1 (in page 37) of 
\cite{Co-M} for the details.) It is known that $Z_G (X,E,Y)$ is a Levi
factor of $Z_G (X)$ (Lemma 3.7.3 in page 50 of \cite{Co-M}).
The center of the Lie algebra of $Z_G (X,E,Y)$ will be denoted by ${\mathfrak c}$.
Let
$$
{\mathfrak c}^{Z_G (X,E,Y)/ Z_G (X,E,Y)^0}\, \subset\, {\mathfrak c}
$$
be the fixed point set for the natural action of
$Z_G (X,E,Y)/ Z_G (X,E,Y)^0$ on ${\mathfrak c}$.

\begin{theorem}\label{nilpotent}
Let $G$ be a complex semisimple Lie group with Lie algebra
$\mathfrak g$. Let $X\, \in\, \mathfrak g$ be a nonzero nilpotent element,
and let ${\mathcal O}_X\,\subset\,
\mathfrak g$ be the orbit of $X$ under the adjoint action.
Let $(X,E,Y)$ be a ${\mathfrak{s} \mathfrak{l}}_2$-triple containing
$X$ as a nil-positive element.
Considering ${\mathcal O}_X$ as a real manifold, there is an isomorphism
$$
H^2 _{\rm{dR}} ({\mathcal O}_X, {\mathbb C})\, :=\,
H^2 _{\rm{dR}} ({\mathcal O}_X, \r) \otimes_\r \c \,=\,
{\mathfrak c}^{Z_G (X,E,Y)/ Z_G (X,E,Y)^0}\, .
$$
\end{theorem}

\begin{proof}
We will apply Theorem
\ref{2nd-cohom-r}. First observe that as $G$ is complex semisimple, any
maximal compact subgroup of $G$ is semisimple. The above defined group $Z_G ( X, E, Y)$
has finitely many connected components, because it is an algebraic group. Also, it
is a Levi subgroup
of $Z_G (X)$. Hence any maximal compact subgroup of $Z_G ( X, E, Y)$ is maximal compact
in $Z_G (X)$.

Let $K$ be a maximal compact subgroup of $Z_G ( X, E, Y)$ with Lie algebra ${\mathfrak k}$.
Thus by Theorem \ref{2nd-cohom-r}, we have
$$
H^2 ( {\mathcal O}_X, \r) \,=\, H^2 ( G / Z_G (X), \r)
\,=\, [{\mathfrak z} ({\mathfrak k})^*]^{K/K^0}\, .
$$
Further, it is easy to see that $[{\mathfrak z} ({\mathfrak k})^*]^{K/K^0} 
\,=\,[{\mathfrak z} ({\mathfrak k})^{K/K^0}]^*$.

Clearly, $Z_{\mathfrak g} ( X, E, Y)$
is the Lie algebra of $Z_G ( X, E, Y)$.
As $K$ is a maximal compact subgroup of $Z_G ( X, E, Y)$, it is immediate that ${\mathfrak k}$ is 
a real form of the complex reductive Lie algebra $Z_{\mathfrak g} ( X, E, Y)$, and moreover, 
${\mathfrak z}({\mathfrak k})$ is a real (sub)form of the center ${\mathfrak c}$ of
$Z_{\mathfrak g} ( X, E, Y)$. Since $K$ is Zariski dense in $Z_G ( X, E, Y)$, the invariant part
${\mathfrak z} ({\mathfrak k})^{K/K^0}$ is a real form of ${\mathfrak c}^{Z_G (X,E,Y)/ Z_G (X,E,Y)^0}$.
Thus, we have
$$ 
H^2 ( {\mathcal O}_X, \r) \otimes_\r \c \,=\, [{\mathfrak z}
({\mathfrak k})^{K/K^0}] \otimes_\r \c \,=\, {\mathfrak c}^{Z_G (X,E,Y)/ Z_G 
(X,E,Y)^0}\, ,
$$
completing the proof of the theorem.
\end{proof}

We shall now recall certain standard notation associated to nilpotent orbits in simple Lie algebras.
For a positive integer $n$, let ${\mathcal P}(n)$ denote the set of partitions of $n$; so
${\mathcal P}(n)$ consists of finite sequence of integers $[ d_1, \cdots , d_k ]$ with 
$\sum_{l=1}^{k} d_l = n$, $0 < d_i$ and $d_i \leq d_{i+1}$ for all $i$. For a partition
${\bf d}\,= \,[ d_1, \cdots , d_k ]\, \in\, {\mathcal P}(n)$, denote by
$|{\bf d}|$ the number of distinct $d_i$. By Theorem 5.1.1 of \cite{Co-M}, if 
$G$ is a complex simple Lie group with Lie algebra
${\mathfrak s}{\mathfrak l}_n$, then the nilpotent orbits are in bijection with ${\mathcal P}(n)$.
For any ${\bf d} \in {\mathcal P}(n)$, the corresponding nilpotent orbits in ${\mathfrak s}{\mathfrak l}_n$
will be denoted by ${\mathcal O}_{\bf d}$. 

For a group $H$, we denote by $H^n_\Delta$ 
the diagonally embedded copy of $H$ in the $n$-fold direct product $H^n$. Let $H_i \subset {\rm GL}_{l_i} $ be a matrix
subgroup, for $ 1 \leq i \leq m$. We then define the subgroup
$$
S( \prod_i H_i) \, :=\, \{ (h_1, \cdots, h_m) \,\in\, \prod_{i=1}^m H_i \, \mid \, \prod_i \det h_i =1\}
\, \subset\, \prod_{i=1}^m H_i\, .
$$

For a positive integer $m$, the cardinality of the set 
$\{j\, \mid \, d_j \,=\, m\}$ is denoted by $r_m ({\bf d})$. Define
$$
{\mathcal P}_1(n) \,:=\, \{{\bf d }\,\in\, {\mathcal P}(n) \,\mid\, r_m ({\bf d}) \,\, {\rm is} \,\, {\rm even}
 \,\, {\rm for} \,\, {\rm all}\,\, {\rm even}\,\, {\rm integers} \,\, m\}\, ,
$$
and
$$
{\mathcal P}_{-1}(2n) \,:=\, \{{\bf d } \,\in\, {\mathcal P}(2n) \,\mid
\, r_m ({\bf d}) \,\, {\rm is} \,\, {\rm even}
 \,\, {\rm for} \,\, {\rm all}\,\, {\rm odd}\,\, {\rm integers} \,\, m\}\, .
$$

\begin{theorem}[Springer-Steinberg]\label{springer-steinberg}
Let $G$ be a simple simply connected complex group of classical type with Lie
algebra ${\mathfrak g}$.
\begin{enumerate}
\item Assume that ${\mathfrak g}= {\mathfrak s}{\mathfrak l}_n$. Take any
${\bf d} \,=\, [d_1, \cdots, d_k]\,\in\,{\mathcal P}(n)$; let ${\mathcal O}_{\bf d}$ be the corresponding
nilpotent orbit. Take any $X \in {\mathcal O}_{\bf d}$. Then for any 
${\mathfrak{s} \mathfrak{l}}_2$-triple of the form
$(X,E,Y)$ in ${\mathfrak g}$,
$$
Z_G (X,E,Y) \,=\, S ( \prod_i ({\rm GL}_{r_i ({\bf d})})^i_\Delta)\, .
$$ 

\item Assume that ${\mathfrak g}= {\mathfrak s}{\mathfrak p}_{2n}$. Take any
${\bf d} \,= \,[d_1, \cdots, d_k] \,\in\, {\mathcal P}_{-1}(2n)$; let ${\mathcal O}_{\bf d}$ be the corresponding
nilpotent orbit. Fix any $X \in {\mathcal O}_{\bf d}$.
Then for any ${\mathfrak{s} \mathfrak{l}}_2$-triple
of the form $(X,E,Y)$ in ${\mathfrak g}$,
$$Z_G (X,E,Y) \,=\, \prod_{i \, {\rm odd}} ({\rm Sp}_{r_i ({\bf 
d})})^i_\Delta
\times \prod_{i \, {\rm even}} ({\rm O}_{r_i ({\bf d})})^i_\Delta\, .
$$ 

\item Assume that ${\mathfrak g}= {\mathfrak s}{\mathfrak o}_{n}$. Take any
${\bf d}\,=\,[d_1, \cdots, d_k] \,\in\,{\mathcal P}_1 (n)$; let ${\mathcal O}_{\bf d}$ be the corresponding
nilpotent orbit. Let $X \in {\mathcal O}_{\bf d}$.
Then for any ${\mathfrak{s} \mathfrak{l}}_2$-triple of the form
$(X,E,Y)$ in ${\mathfrak g}$,
$$Z_G (X,E,Y)\, = \, \, {\rm double} \,\, {\rm cover} \,\, {\rm of} \,\,
 S ( \prod_{i \, {\rm even}} ({\rm Sp}_{r_i
({\bf d})})^i_\Delta \times \prod_{i \, {\rm odd}} ({\rm O}_{r_i ({\bf d})})^i_\Delta)\, .
$$
\end{enumerate}
\end{theorem}

See Theorem 6.1.3 of \cite{Co-M} for Theorem \ref{springer-steinberg}.

\begin{theorem}\label{a_n}
Let $G$ be a complex simple group with Lie algebra ${\mathfrak s}{\mathfrak l}_n$.
For any ${\bf d} \,\in\, {\mathcal P}(n)$,
$$
\dim H^2 ({\mathcal O}_{\bf d}, \r)\, =\, |{\bf d}|-1\, .
$$
\end{theorem}

\begin{proof}
We may assume that $G$ is simply connected. Let $| {\bf d}| \,=\, \alpha$.
Fix an element $X \,\in\, {\mathcal O}_{\bf d}$.
For any ${\mathfrak s} {\mathfrak l}_2$-triple
$(X,E,Y)$ in ${\mathfrak g}$, using Theorem \ref{springer-steinberg}, we have
$$
Z_G (X,E,Y) \,=\,S ( \prod_{i=1}^{\alpha} ({\rm GL}_{r_i ({\bf 
d})})^i_\Delta)\, .
$$ 
Define
$$
{\bf H} \,:=\, \{ (A_1, \cdots A_\alpha) \,\in\, \prod_{i=1}^{\alpha} {\rm GL}_{r_i ({\bf d})} \,
\mid\, \prod_ {i=1}^{\alpha} (\det A_i)^i =1\}\, ,
$$ 
and let ${\mathfrak h}$ be the Lie algebra of ${\bf H}$.
Clearly, ${\bf H}\,=\, S ( \prod_{i=1}^{\alpha} ({\rm GL}_{r_i ({\bf 
d})})^i_\Delta)$. Define
$L \,:= \,\prod_{i=1}^{\alpha} {\rm GL}_{r_i ({\bf d})}$, and let
${\mathfrak l}$ be the Lie algebra of $L$. Then $[L\, ,L] \,=\, 
\prod_{i=1}^{\alpha} SL_{r_i ({\bf d})}$, and $$[L\, ,L] \,\subset \,{\bf H} 
\,\subset
\, L\, .$$ Consequently, 
$[{\mathfrak l},{\mathfrak l}] \,\subset\, {\mathfrak h} \,\subset \,{\mathfrak l}$.
Hence 
$${\mathfrak z}({\mathfrak h}) \,=\, {\mathfrak h} \cap {\mathfrak 
z}({\mathfrak l}),\,~ [{\mathfrak l},{\mathfrak l}] \,=\,
[{\mathfrak h},{\mathfrak h}] ~\, \text{~and~}\,~ 
{\mathfrak h} \,=\, ({\mathfrak h} \cap {\mathfrak z}({\mathfrak 
l}))\oplus [{\mathfrak l},{\mathfrak l}]\, .$$
Now adjoint action of $L$ fixes all of ${\mathfrak z}({\mathfrak l})$. So the adjoint action of ${\bf H}$ on
${\mathfrak h} \cap {\mathfrak z}({\mathfrak l})={\mathfrak z}({\mathfrak h})$ is trivial.
In particular, ${\mathfrak z}({\mathfrak h})^{{\bf H}/ {\bf H}^0} = {\mathfrak z}({\mathfrak h})$. 
Thus, appealing to Theorem
\ref{nilpotent}, it remains to find the dimension of ${\mathfrak z}({\mathfrak h})$. As 
${\mathfrak z}({\mathfrak h}) = {\mathfrak h} \cap {\mathfrak z}({\mathfrak l})$, it is easy to check that
$$
{\mathfrak z}({\mathfrak h}) \,=\,
\{(x_1, \cdots x_\alpha) \,\in\, \c^\alpha \,\mid\, \sum_{i=1}^{\alpha} i r_i ({\bf d}) x_i \,=\,0 \}\, .
$$
Thus $\dim {\mathfrak z}({\mathfrak h}) = \alpha -1 = | {\bf d}|-1$.
This completes the proof.
\end{proof}

It follows from Theorem 5.1.6 of \cite{Co-M} that the nilpotent orbits in 
${\mathfrak s}{\mathfrak p}_{2n}$ are in bijection with ${\mathcal 
P}_{-1}(2n)$. As before, for any ${\bf d} \in {\mathcal P}_{-1} (2n)$, 
the corresponding nilpotent orbits in 
${\mathfrak s}{\mathfrak p}_{2n}$ will be denoted by ${\mathcal O}_{\bf d}$. 

\begin{theorem}\label{c_n}
Let $G$ be a complex simple group with Lie algebra ${\mathfrak s}{\mathfrak p}_{2n}$.
Let $X$ be a nilpotent element in ${\mathfrak s}{\mathfrak p}_{2n}$. Then 
$$
H^2 ({\mathcal O}_X, \r) \,=\, 0\, .
$$
\end{theorem}

\begin{proof} 
We may assume that $G$ is simply connected. Take any
$${\bf d} \,=\, [d_1, \cdots, d_k] \,\in\, {\mathcal P}_{-1}(2n)\, ,$$ and let
${\mathcal O}_{\bf d}$ be the corresponding
nilpotent orbit in ${\mathfrak s}{\mathfrak p}_{2n}$. Fix an element $X \,\in\, {\mathcal O}_{\bf d}$.
For any ${\mathfrak s} {\mathfrak l}_2$-triple
$(X,E,Y)$ in ${\mathfrak g}$, using Theorem \ref{springer-steinberg}, we have
$$
Z_G (X,E,Y) \,=\, \prod_{i \, {\rm odd}} ({\rm Sp}_{r_i
({\bf d})})^i_\Delta \times \prod_{i \, {\rm even}} ({\rm O}_{r_i ({\bf d})})^i_\Delta\, .
$$ 
Clearly 
$$\prod_{i \, {\rm odd}} ({\rm Sp}_{r_i ({\bf d})})^i_\Delta \times \prod_{i \, 
{\rm even}} ({\rm O}_{r_i ({\bf d})})^i_\Delta\,= \, \prod_{i \, {\rm odd}} {\rm 
Sp}_{r_i ({\bf d})} \times \prod_{i \, {\rm even}} {\rm O}_{r_i ({\bf d})}\, .
$$ Define $${\bf H} \,:=\,\prod_{i \, {\rm odd}} {\rm Sp}_{r_i ({\bf d})}
\times \prod_{i \, {\rm even}} {\rm O}_{r_i ({\bf d})}\, ,$$ and let ${\mathfrak h}$
be the Lie algebra of ${\bf H}$. Observe that the
connected component of ${\bf H}$ is semisimple
unless there is an even integer $i$ with $r_i ({\bf d})\,=\,2$. So it 
only remains to consider the case
where
$$E \,:= \,\{i ~ \,{\rm even} \, \mid \, r_i ({\bf d})\,=\, 2\}$$ is non-empty. 
Now, in this case,
$$
 {\mathfrak z}({\mathfrak h}) \,= \, \bigoplus_{i \in E} {\mathfrak 
s}{\mathfrak o}_2\, .
$$
As ${\rm O}_2/ {\rm SO}_2 \,=\, {\mathbb Z}/2{\mathbb Z}$
acts non-trivially on ${\mathfrak s}{\mathfrak o}_2$, it follows that 
$[{\mathfrak z}({\mathfrak h})]^{{\bf H}/ {\bf H}^0} \,=\,0$. Now the proof is completed by Theorem
\ref{nilpotent}.
\end{proof}

The nilpotent orbits in ${\mathfrak s}{\mathfrak o}_n$
are in bijection with ${\mathcal P}_1(n)$ (see Theorem 5.1.6 of \cite{Co-M}). 
As before, for any ${\bf d} \in {\mathcal P}_1(n)$, the corresponding nilpotent orbits in 
${\mathfrak s}{\mathfrak o}_n$ will be denoted by ${\mathcal O}_{\bf d}$. 

\begin{theorem}\label{bd_n}
Let $G$ be a complex simple Lie group with Lie algebra ${\mathfrak s}{\mathfrak 
o}_n$.
take any ${\bf d} \,\in \,{\mathcal P}_1(n)$.
\begin{enumerate}
\item Assume that ${\bf d}$ is such that there an odd integer $m$ with $r_m ({\bf d})\,=\, 2$,
and $r_l ({\bf d})\,=\,0$ for all odd integers $l \,\neq\, m$. Then 
$$
\dim H^2 ({\mathcal O}_{\bf d}, \r) \,=\,1\, .
$$

\item If ${\bf d}$ does not satisfy the above condition, then
$$
\dim H^2 ({\mathcal O}_{\bf d}, \r) \,=\,0\, .
$$
\end{enumerate}
\end{theorem}

\begin{proof}
We may assume that $G\,=\, {\rm SO}_n$.
Let ${\bf d}\,=\, [d_1, \cdots, d_k]\,\in\,{\mathcal P}_1 (n)$, and let 
${\mathcal O}_{\bf d}$
be the corresponding nilpotent orbit. Fix $X \,\in\, {\mathcal O}_{\bf d}$.
For any ${\mathfrak s} {\mathfrak l}_2$-triple
$(X,E,Y)$ in ${\mathfrak g}$, using Theorem \ref{springer-steinberg}, we have
$$
Z_G (X,E,Y) \,=\,
S ( \prod_{i \, {\rm even}} ({\rm Sp}_{r_i ({\bf d})})^i_\Delta \times \prod_{i \, {\rm odd}}
({\rm O}_{r_i ({\bf d})})^i_\Delta)\, .
$$

Further note that 
$$S ( \prod_{i \, {\rm even}} ({\rm Sp}_{r_i ({\bf d})})^i_\Delta \times \prod_{i \, {\rm odd}}
({\rm O}_{r_i ({\bf d})})^i_\Delta)\,=\,
\prod_{i \, {\rm even}} {\rm Sp}_{r_i ({\bf d})} \times S(\prod_{i \, {\rm odd}}
({\rm O}_{r_i ({\bf d})})^i_\Delta))\, .
$$

Define ${\bf H} := \prod_{i \, {\rm even}} {\rm Sp}_{r_i ({\bf d})} \times
S(\prod_{i \, {\rm odd}} ({\rm O}_{r_i ({\bf d})})^i_\Delta)$; let
${\mathfrak h}$ be the Lie algebra of ${\bf H}$.

{\it Proof of (1):} Assume that there an odd integer $m$ with $r_m ({\bf d})=2$
and $r_l ({\bf d})=0$, for all odd integers $l \neq m$. Hence
$$
{\bf H} \,= \,
\prod_{i \, {\rm even}} {\rm Sp}_{r_i ({\bf d})} \times S( ({\rm O}_2)^m_\Delta))
\,=\, \prod_{i \, {\rm even}} {\rm Sp}_{r_i ({\bf d})} \times {\rm SO}_2\, .
$$

Clearly $Z({\bf H}) \,=\, {\rm SO}_2$, and ${\mathfrak z}({\mathfrak h})= 
{\mathfrak s}{\mathfrak o}_2$. Hence
$[{\mathfrak z}({\mathfrak h})]^{{\bf H}/ {\bf H}^0} \,=\, {\mathfrak 
z}({\mathfrak h})\,=\,
{\mathfrak s}{\mathfrak o}_2$. Thus $$\dim [{\mathfrak z}({\mathfrak h})]^{{\bf H}/ {\bf H}^0} \,=\,1\, .$$
Now the proof follows from Theorem \ref{nilpotent}.  

{\it Proof of (2):} Part (2) is broken into two sub-cases.

{\it Case 1:} Assume that $r_m ({\bf d})\,\neq\, 2$ for all odd integer $m$.
Then we have
$$
{\bf H} \,=\, 
\prod_{i \, {\rm even}}
{\rm Sp}_{r_i ({\bf d})} \times S(\prod_{i \, {\rm odd}, \, r_i ({\bf d})=1 }
({\rm O}_{r_i ({\bf d})})^i_\Delta
\times \prod_{i \, {\rm odd}, \, r_i ({\bf d}) >2 } ({\rm O}_{r_i ({\bf d})})^i_\Delta)\, .
$$
Observe that the connected component ${\bf H}^0$ is semisimple; hence ${\mathfrak z}({\mathfrak h})\,=\,0$.
We use Theorem
\ref{nilpotent} to complete the proof. 

{\it Case 2:} Define ${\mathcal F}\, := \, \{i \,\, {\rm is} \,\, {\rm odd}\, \mid\, 
 r_m ({\bf d})\,=\,2\}$ and
$$
{\mathcal F}^* \,:=\, \{i \,\, {\rm is} \,\, {\rm odd}\, \mid 
\, r_m ({\bf d})\,\neq \,2, \, \, r_m ({\bf d}) \,\neq\, 0\}\, .
$$
In this second case we assume that either $|{\mathcal F}| \,\geq\, 2$ or $|{\mathcal F}| \,=\,1\, \leq \, |{\mathcal F}^*|$.

Let $\beta = |{\mathcal F}|$ and $\gamma = |{\mathcal F}^*|$. We enumerate elements of ${\mathcal F}$ as ${\mathcal F} = \{ i_1, \cdots, i_\beta \}$,
and those of ${\mathcal F}^*$ as ${\mathcal F}^* = \{ j_1, \cdots, j_\gamma \}$.
Define the group $$L \,:= \, \prod_{p=1}^{\beta} {\rm O}_{r_{i_p} ({\bf d})}
\times \prod_{q=1}^{\gamma} {\rm O}_{r_{i_q} ({\bf d})}\, . $$
We identify
$S(\prod_{i \in {\mathcal F}} ({\rm O}_{r_i ({\bf d})})^i_\Delta
\times \prod_{i \in {\mathcal F}^*} ({\rm O}_{r_i ({\bf d})})^i_\Delta)$ with the subgroup $M$ of $L$ 
defined as follows:
$$
M \,:=\, \{(A_1, \cdots, A_\alpha, B_1, \cdots B_\beta)
$$
$$
\in\, \prod_{p=1}^{\beta} {\rm O}_{r_{i_p} ({\bf d})}
\times \prod_{q=1}^{\gamma} {\rm O}_{r_{i_q} ({\bf d})} \,\mid\, \prod_{p=1}^{\beta} (\det A_p)^{i_p} 
\prod_{q=1}^{\gamma} (\det B_q)^{j_q} = 1\}\, .
$$
Let ${\mathfrak m}$ be the Lie algebra of $M$.
Clearly we have
$$
{\bf H} \,=\, 
\prod_{i \, {\rm even}} {\rm Sp}_{r_i ({\bf d})} \times M\, .
$$

As ${\rm O}_{r_{i_p}({\bf d})} \,=\, {\rm O}_2$ for all $p = 1, \cdots , \alpha$,
and ${\rm O}_{r_{i_p}({\bf d})} \,\neq\,{\rm O}_2$ for all $q = 1, \cdots , \beta$,
it follows that 
$$
{\mathfrak z}({\mathfrak h}) \,=\, {\mathfrak z}({\mathfrak m}) \,=
\, \bigoplus_{p=1}^{\alpha} {\mathfrak s}{\mathfrak o}_2\, .
$$

Let $\pi_p : L \,\longrightarrow\, {\rm O}_{r_{i_p} ({\bf d})} \,=\, {\rm O}_2$
be the projection of $L$ onto the $p$-th factor ${\rm O}_{r_{i_p} ({\bf d})}$.
Now as either $|{\mathcal F}| \geq 2$ or $|{\mathcal F}| =1$ and $|{\mathcal F}^*| \geq 1$, it is easy to see that 
$\pi_p (M) \,=\,{\rm O}_{r_{i_p} ({\bf d})}\,=\,{\rm O}_2$ for all $p= 1, \cdots , \alpha$. 
We identify $\bigoplus_{p=1}^{\alpha} 
{\mathfrak s}{\mathfrak o}_2$ canonically with a subalgebra of ${\mathfrak m}$. 
Since ${\rm O}_2/{\rm SO}_2 \,=\, {\mathbb Z}/2{\mathbb Z}$
acts nontrivially on ${\mathfrak s}{\mathfrak o}_2$, it follows that the adjoint action of $M$ has no
nontrivial fixed points in $\bigoplus_{p=1}^{\alpha} 
{\mathfrak s}{\mathfrak o}_2$. Thus $[{\mathfrak z}({\mathfrak h})]^{{\bf H}/ {\bf H}^0} \,=\,0$. In view of
Theorem
\ref{nilpotent} this
completes the proof.
\end{proof}

We next put down some results that will facilitate investigating the 
nilpotent orbits in the exceptional Lie algebras.

\begin{lemma}\label{outer}
Let $G$ be a complex algebraic group with Lie algebra ${\mathfrak g}$, and let
$H$ be a complex semisimple algebraic subgroup of $G$ with Lie algebra
${\mathfrak h}$. Let $c \in G$ be a semisimple element such that $c H c^{-1} \,=\,
H$. Let ${\rm Ad} (c)\vert_{\mathfrak h}$ be the restriction of
${\rm Ad} (c) \,\in\, {\rm Aut} ({\mathfrak g})$.
Then ${\rm Ad} (c)\vert_{\mathfrak h} \,\in\, {\rm Ad}(H)$ if and only if there is a Cartan
subalgebra ${\mathfrak t}$ of ${\mathfrak h}$ fixed pointwise by ${\rm 
Ad}(c)$.
\end{lemma}

\begin{proof}
Suppose ${\rm Ad} (c)\vert_{\mathfrak h} \,=\, {\rm Ad}(h)$ for some $h \,\in\, H$.
As $c$ is a semisimple element, and $H$ is a semisimple group, it follows that 
$h$ is a semisimple element in $H$. Consider a maximal torus $T$ of $H$
such that $h \in T$. Then the Lie algebra ${\mathfrak t}$ of $T$ satisfies the
condition that ${\rm Ad} (c) (x)\, =\, x$ for all $x \,\in\, {\mathfrak t}$.

Conversely, suppose ${\rm Ad} (c) (x)\, =\, x$ for all $x \in 
{\mathfrak t}$, where ${\mathfrak t}\, \subset\, \mathfrak h$ is some 
Cartan subalgebra. Then ${\rm Ad} (c)\vert_{\mathfrak h} \,\in\, {\rm 
Ad}(H)$ by Proposition 3 (in Ch. IX) of \cite{Ja}.
\end{proof}

\begin{lemma}\label{dim} Let $G$ be a complex algebraic group. 
Let $H$ be a connected complex semisimple subgroup of $G$ with Lie algebra ${\mathfrak h}$.
Let $c \in G$ be a semisimple element such that $c H c^{-1} \,=\, H$. There is a Cartan
subalgebra ${\mathfrak t}$ of ${\mathfrak h}$ such that the following two hold:
\begin{enumerate}
\item Firstly, ${\rm Ad }(c) {\mathfrak t}\,=\, {\mathfrak t}$.
There is a set of simple roots, say $\Delta$, with respect to ${\mathfrak t}$, 
such that the action of $c$
on ${\mathfrak t}^*$ keeps $\Delta$ invariant, that is, $ \lambda \circ {\rm Ad}(c)
\,\in\, \Delta$, for all $\lambda \in \Delta$.

\item Let $\Gamma \,:=\, \langle c \rangle$ denote the group generated by $c$, and let 
${\mathfrak t}^c\, \subset\, \mathfrak t$ be the fixed point set for the action of $c$.
Let $m$ be the number of distinct $\Gamma$--orbits in $\Delta$. Then 
$\dim {\mathfrak t}^c \,\geq\, m$.
\end{enumerate}
\end{lemma}

\begin{proof}
We first prove (1). By Theorem 7.5 of \cite{St}, there exists a Borel subgroup
$B$ of $G$, and maximal torus $T\,\subset\, B$, such that $cTc^{-1} \,= \,T$ and $cBc^{-1} \,=\, B$.
Let ${\mathfrak t}$ and ${\mathfrak b}$ be the Lie algebras of $T$ and $B$ respectively.
Then ${\rm Ad }(c) ({\mathfrak t})\,=\, {\mathfrak t}$, and
${\rm Ad }(c) ({\mathfrak b})\,=\, {\mathfrak b}$. Consider the root system of ${\mathfrak h}$ with respect to 
${\mathfrak t}$; let $\Delta$ be the set of simple roots in the set of positive roots
given by ${\mathfrak b}$. As
${\rm Ad }(c) ({\mathfrak b})\,=\, {\mathfrak b}$, we have $ \lambda \circ {\rm Ad}(c) \,\in\, \Delta$ 
for all $\lambda \,\in\, \Delta$.

We will now prove (2). Let $\Delta_i$, $i = 1, \cdots, m $, be the $m$ distinct orbits in $\Delta$ under the action of $\Gamma$.
For each $ i \,\in\, \{ 1, \cdots, m \}$, define 
$$
V_i \,:=\, \{ x \in {\mathfrak t} \, \mid\, \lambda (x) = \mu (x) \,\, {\rm for } \,\, {\rm all }
\,\,\lambda,\mu \in \Delta_i, \,\, {\rm and } \,\,
\delta (x) = 0 \,\, {\rm for } \,\, {\rm all } \,\, \delta \in \Delta - \Delta_i \}\, .
$$
As $\Delta$ is a basis for ${\mathfrak t}^*$, it follows
that $V_i \subset {\mathfrak t}^c$, $V_i \cap V_j =0$ for all
$i \neq j$, and that $\dim V_i \,=\,1$ for all $i$.
Thus $\dim {\mathfrak t}^c \,\geq\, \dim \sum_{i=1}^m V_i \,= \,m$.
\end{proof}

\begin{lemma}\label{new-dim} Let $G$, $H$, ${\mathfrak h}$, $c$, $\Gamma$ be 
as in the Lemma \ref{dim}. Suppose
further that $\Gamma \,=\, {\mathbb Z}/2{\mathbb Z}$. 
\begin{enumerate}
\item Assume that ${\mathfrak h} \,=\, {\mathfrak a}_n$. Then there is 
toral subalgebra ${\mathfrak t}$ of
${\mathfrak h}$ such that ${\rm Ad} (c)(x)\,
 \,=\,x$ for all $x \,\in\, {\mathfrak t}$; moreover, $\dim {\mathfrak 
t} 
\,\geq\, n/2$ if $n$ is even
and $\dim {\mathfrak t} \,\geq\, (n+1)/2$ if $n$ is odd. 

\item Assume that ${\mathfrak h} \,=\, {\mathfrak b}_n$. Then there is toral 
subalgebra ${\mathfrak t}$ of ${\mathfrak h}$ such that ${\rm Ad} 
(c)(x)\, \,=\,x$ for all $x \,\in\, {\mathfrak t}$; moreover, $\dim 
{\mathfrak t}\,= \,n$.
\end{enumerate}
\end{lemma}

\begin{proof} We retain the notation used in the
statement and proof of Lemma \ref{dim}. 

We first prove (1). 
If ${\rm Ad} (c)\vert_{\mathfrak h} \,\in\, {\rm Ad}(H)$ then using 
Lemma \ref{outer}
we see that there is a Cartan subalgebra ${\mathfrak t}$ of ${\mathfrak 
h}$ which remains fixed by the adjoint action of $c$. 

Assume that the restriction ${\rm Ad} (c)\vert_{\mathfrak h} \,\notin \, 
{\rm Ad}(H)$. Now by part (1) of Lemma \ref{dim}
we consider a Cartan subalgebra ${\mathfrak t}$ of ${\mathfrak h}$ and 
a set of simple roots $\Delta$ which remain
invariant under the action of $c$. Appealing to Lemma \ref{outer} we 
see that the action of $c$ on $\Delta$ must be nontrivial. From
the Dynkin diagram of ${\mathfrak a}_n$ we see that 
if ${\mathfrak h} \,=\, {\mathfrak a}_n$, and if $n$ is even, then
the number of orbits in $\Delta$, under the nontrivial action of 
$\Gamma$, is $n/2$. Similarly, if ${\mathfrak h}\,=\, {\mathfrak a}_n$ 
and if $n$ is odd then the number of orbits in $\Delta$, under the 
nontrivial action of $\Gamma$, is $(n+1)/2$.
Now the proof follows from (2) of Lemma \ref{dim}.

We now prove (2).
Let now ${\mathfrak h} \,=\, {\mathfrak b}_n$. As before, by part (1) of 
Lemma \ref{dim}, 
we consider a Cartan subalgebra ${\mathfrak t}$ of ${\mathfrak h}$ and a set of simple roots $\Delta$ which remain
invariant under the action of $c$. But
there are no nontrivial graph-automorphism of the Dynkin diagram of 
${\mathfrak b}_n$; in particular, $\Gamma$-action fixes $\Delta$. Thus
$\Gamma$-action fixes ${\mathfrak t}$.
\end{proof}

\begin{lemma}\label{toral} Let $\Gamma$ be a nontrivial
finite abelian group acting on a complex torus $T$ by automorphisms.
Assume that all the maximal diagonalizable subgroups of the algebraic group
$\Gamma \ltimes T$ are either finite or coincide with $T$.
Then the fixed point set of $T$ under
the action of $\Gamma$ is finite.
\end{lemma}

\begin{proof} 
We need to show that $\Gamma$ does not commute with a sub-torus of $T$ of positive dimension. To prove
by contradiction, suppose that 
$\Gamma$ centralizes a sub-torus $T^\prime \,\subset\, T$ of positive dimension. 

The direct product $\Gamma \times T^\prime$ is an abelian subgroup of
$\Gamma \ltimes T$, and it consists entirely of semisimple elements;
so $\Gamma \times T^\prime$ is diagonalizable.
As maximal diagonalizable subgroups of $\Gamma \ltimes T$ are either $T$ or finite, it follows that a
$\Gamma \ltimes T$-conjugate of 
$\Gamma \times T^\prime$ is contained in either $T$ or in some finite subgroup.
This leads to a contradiction, because $\Gamma$ does not lie in a connected subgroup of
$\Gamma \ltimes T$ and $T^\prime$ is not a finite group. Therefore,
$\Gamma$ does not commute with a sub-torus of $T$ of positive dimension.
\end{proof}

We now deal with the exceptional Lie algebras and 
henceforth, for notational convenience, by a nilpotent element of an
exceptional Lie algebra we will mean a nonzero nilpotent element.

A toral algebra of dimension $k$ is denoted by ${\mathfrak t}_k$
while a torus of dimension $k$ is denoted by $T_k$. 

Let $G$ be a simple group with Lie algebra ${\mathfrak g}$.
For a nilpotent element $X \in {\mathfrak g}$ and a ${\mathfrak s}{\mathfrak l}_2$-triple
$(X, E, Y)$ recall that $Z_{\mathfrak g} (X,E,Y)$ is
the Lie algebra of $Z_G (X,E,Y)$.
The component group $Z_G (X,E,Y)/ Z_G (X,E,Y)^0$ will be denoted by $\Gamma$.

Among the exceptional simple Lie algebras, the case of
${\mathfrak e}_6$ will be treated separately.

\begin{theorem}\label{gfe}
Let $\mathfrak g$ be one of ${\mathfrak g}_2$, ${\mathfrak f}_4$, ${\mathfrak e}_7$
and ${\mathfrak e}_8$. Let $G$ be a complex simple Lie group with Lie algebra $\mathfrak g$.
The orbit of a nilpotent element $X \,\in\, {\mathfrak g}$ for the adjoint action of
$G$ on $\mathfrak g$ will be denoted by ${\mathcal O}_X$. Then 
$$
H^2 ({\mathcal O}_X, \r)\,=\,0\, .
$$
\end{theorem}

\begin{proof}
We may assume that $G$ is of adjoint type.

{\bf Proof for ${\mathfrak g} = {\mathfrak g}_2$:}\,
In this case there are total $4$ nilpotent orbits (see Table 1 of \cite{A}).
Let $X \,\in\, {\mathfrak g}_2$ be a nilpotent element.
Let $(X,E,Y)$ be a ${\mathfrak s} {\mathfrak l}_2$-triple
in ${\mathfrak g}_2$ containing $X$. From Table 1, Column 4 of \cite{A} we
see that the Lie algebra $Z_{{\mathfrak g}_2} (X,E,Y)$ is either trivial 
or it is simple of type ${\mathfrak a}_1$.
Thus the center of $Z_{{\mathfrak g}_2} (X,E,Y)$ is of dimensional zero.
So, by Theorem \ref{nilpotent} we conclude that $H^2 ( {\mathcal O}_X, \r) =0$.

{\bf Proof for ${\mathfrak g} = {\mathfrak f}_4$:}\, In this case
there are total $15$ nilpotent orbits (see
Table 2 of \cite{A}).
Let $X \,\in\, {\mathfrak f}_4$ be a nilpotent element.
Let $(X,E,Y)$ be a ${\mathfrak s} {\mathfrak l}_2$-triple
in ${\mathfrak f}_4$ containing $X$. From Table 2, Column 4 of \cite{A} we 
see that the Lie algebra $Z_{{\mathfrak f}_4} (X,E,Y)$ is either trivial 
or one of the following:
$$
{\mathfrak c}_3\, , \, {\mathfrak a}_3\, ,\, {\mathfrak a}_1 \oplus 
{\mathfrak a}_1\, ,\,
{\mathfrak a}_2\, , \, {\mathfrak g}_2\, , \, {\mathfrak a}_1\, .
$$
Thus the center of $ Z_{{\mathfrak f}_4} (X,E,Y)$ is zero dimensional.
So, from Theorem \ref{nilpotent} it follows that $H^2 ( {\mathcal O}_X, 
\r) \,=\,0$.

{\bf Proof for ${\mathfrak g} ={\mathfrak e}_7$:}\,
In this case there are $44$ nilpotent orbits in ${\mathfrak e}_7$. Let 
$X \,\in\, {\mathfrak e}_7$ be a nilpotent element (see Table 4 of \cite{A}).
Let $(X,E,Y)$ be a ${\mathfrak s} {\mathfrak l}_2$-triple
in ${\mathfrak e}_7$ containing $X$. From Table 4, Column 4 in \cite{A} we 
see that out of $44$ distinct conjugacy classes of nilpotent
elements, $38$ nilpotent orbits ${\mathcal O}_X$ have the property that 
$Z_{{\mathfrak e}_7} (X,E,Y)$ is either trivial or one of the following:
$${\mathfrak d}_6 ,\, {\mathfrak b}_4 \oplus {\mathfrak a}_1,\, {\mathfrak c}_3 \oplus {\mathfrak a}_1,\,
{\mathfrak f}_4,\, {\mathfrak a}_5, \,{\mathfrak c}_3,\,
{\mathfrak a}_1\oplus {\mathfrak a}_1 \oplus {\mathfrak a}_1, \,{\mathfrak g}_2,\,
{\mathfrak g}_2 \oplus {\mathfrak a}_1,\, {\mathfrak b}_3 \oplus {\mathfrak a}_1,\, {\mathfrak b}_3, \,
{\mathfrak a}_1\oplus {\mathfrak a}_1,\, {\mathfrak a}_1, \,{\mathfrak b}_2\, .
$$
As all the above Lie algebras are semisimple, for these $38$ nilpotent 
conjugacy classes
${\mathcal O}_X$, the centers $Z_{{\mathfrak e}_7} (X,E,Y)$ are trivial. 
So, by Theorem \ref{nilpotent},
for these $38$ nilpotent conjugacy classes we have $H^2 ({\mathcal O}_X, \r)\,=\,0$.

We now consider the remaining $6$ nilpotent orbits ${\mathcal O}_X$ for which the center of the 
Lie algebra $Z_{{\mathfrak e}_7} (X,E,Y)$ is nontrivial. We list $Z_{{\mathfrak e}_7} (X,E,Y)$ for
these $6$ remaining nilpotent orbits:
$$
{\mathfrak a}_3 \oplus {\mathfrak t}_1\, ,\, {\mathfrak a}_1 \oplus {\mathfrak t}_1\, ,\,
{\mathfrak a}_2 \oplus {\mathfrak t}_1\, ,\, {\mathfrak t}_2\, ,\, {\mathfrak 
t}_1\, .
$$

In all these 6 cases we note certain common features. From
Table 4, Column 6, Rows 7, 19, 21, 24, 26, 38 of \cite{A} it 
follows that the exact sequence
$$
1 \longrightarrow Z_G (X,E,Y)^0 \longrightarrow Z_G (X,E,Y) \longrightarrow \Gamma \longrightarrow 1
$$
splits and consequently,
$$Z_G (X,E,Y) \,= \, \Gamma \ltimes Z_G (X,E,Y)^0\,. $$
In particular, $\Gamma$ can be identified with a subgroup of $Z_G (X,E,Y)$.
Moreover, from Table 4, Column 5, Rows 7, 19, 21, 24, 26, 38 of \cite{A}
we record that $\Gamma\,=\, {\mathbb Z}/2{\mathbb Z}$. The generator of 
$\Gamma$ will be denoted by $c$.

{\it The case of $Z_{{\mathfrak e}_7} (X,E,Y) = {\mathfrak a}_3 \oplus 
{\mathfrak t}_1$:}\, We refer to
Row 7 in Table 4 of \cite{A} for relevant facts on this case. Clearly, 
as ${\mathfrak a}_3$ is the semisimple part of ${\mathfrak a}_3 \oplus 
{\mathfrak t}_1$, the group
$\Gamma$ normalizes ${\mathfrak a}_3$. In other words, ${\rm Ad}(c) 
{\mathfrak a}_3\,=\,{\mathfrak a}_3$.
Now we use part (1) of Lemma \ref{new-dim} to see that ${\mathfrak a}_3$ 
admits a toral subalgebra, say ${\mathfrak t}$, 
with $\dim {\mathfrak t} = 2$ such that ${\rm Ad}(c) x = x$, for all $x 
\in {\mathfrak t}$.
But the maximum possible dimension of a toral subalgebra of ${\mathfrak a}_3 \oplus {\mathfrak t}_1$ 
centralized by $\Gamma$ is two (see Table 4, Column 3, Row 7 of \cite{A}). Consequently,
$\Gamma$ acts nontrivially on the one dimensional center ${\mathfrak t}_1$ of ${\mathfrak a}_3 \oplus {\mathfrak t}_1$.
Hence by Theorem \ref{nilpotent} $H^2 ({\mathcal O}_X, \r)\,=\,0$. 
 
{\it The case of $Z_{{\mathfrak e}_7} (X,E,Y) = {\mathfrak a}_1 \oplus {\mathfrak t}_1$:}\,
This appears two times in Row 19 and also in Row 26 of Table 4 in \cite{A}. We deal with both
the cases simultaneously.
As in the above case, we have that ${\rm Ad}(c) {\mathfrak a}_1\,=\, 
{\mathfrak a}_1$.
Using part (1) of Lemma \ref{new-dim} we conclude that there is toral 
subalgebra ${\mathfrak t} \subset {\mathfrak a}_1$, with 
$\dim {\mathfrak t} \,=\, 1$, such that ${\rm Ad}(c) x = x$, for all $x 
\in 
{\mathfrak t}$.
But from Table 4, Column 3, Rows 19, 26 of \cite{A} 
we see that the maximum possible dimension of a toral subalgebra of ${\mathfrak a}_3 
\oplus {\mathfrak t}_1$ which is
centralized by $\Gamma$ is one. Thus $\Gamma$ acts nontrivially on the one dimensional center ${\mathfrak t}_1$ of 
${\mathfrak a}_1 \oplus {\mathfrak t}_1$. Hence applying Theorem \ref{nilpotent}
we conclude that $H^2 ({\mathcal O}_X, \r)=0$.

{\it The case of $Z_{{\mathfrak e}_7} (X,E,Y) = {\mathfrak a}_2 \oplus 
{\mathfrak t}_1$:}\, See Row 21 of Table 4 in
\cite{A} for the required facts on this case. 
Just as above, we see that ${\rm Ad}(c) {\mathfrak a}_2 \,=\, {\mathfrak 
a}_2$. Appealing to part (1) of Lemma \ref{new-dim} we conclude that 
${\mathfrak a}_2$ admits a toral subalgebra ${\mathfrak t}$, 
with $\dim {\mathfrak t} \,=\, 1$, such that ${\rm Ad}(c) x \,=\, x$, 
for all $x \,\in \,{\mathfrak t}$.
 Recall the fact (see Table 4, Column 3, Row 21 in \cite{A}) that 
the maximum possible dimension of a toral subalgebra of ${\mathfrak a}_2 \oplus {\mathfrak t}_1$ 
centralized by $\Gamma$ is one. Thus $\Gamma$ must act nontrivially on the one dimensional 
center ${\mathfrak t}_1$ of ${\mathfrak a}_2 \oplus {\mathfrak t}_1$.
Therefore, $H^2 ({\mathcal O}_X, \r)\,=0\,$ by Theorem \ref{nilpotent}. 
 
{\it The case of $Z_{{\mathfrak e}_7} (X,E,Y) = {\mathfrak t}_2$:}\, For this case
we refer to Row 24 in Table 4 of \cite{A} for the required results.
Clearly, $Z_G (X,E,Y)^0 \,=\,T_2$, and $Z_G (X,E,Y) \,=\, \Gamma \ltimes 
T_2$. The maximal diagonalizable subgroups of $Z_G (X,E,Y)$ are isomorphic to 
either ${\mathbb Z}/ 4{\mathbb Z}$
or $T_2$; see Column 3, Row 21 in Table 4 of \cite{A}. From Lemma 
\ref{toral} it follows that 
$\Gamma$ has no nontrivial fixed points in the toral algebra ${\mathfrak t}_2$.
Hence, $H^2 ({\mathcal O}_X, \r)\,=\,0$ by Theorem \ref{nilpotent}.

{\it The case of $Z_{{\mathfrak e}_7} (X,E,Y) = {\mathfrak t}_1$:}\, We refer 
to Row 38 in Table 4 of \cite{A} for this case. 
Note that $Z_G (X,E,Y)^0 \,=\, T_1$, and $Z_G (X,E,Y) \,=\, \Gamma 
\ltimes T_1$. Using Column 3, Row 21 in Table 4 of \cite{A} we conclude 
that the maximal diagonalizable subgroups of $Z_G (X,E,Y)$
are of the form either ${\mathbb Z}/2{\mathbb Z}$
or $T_1$. Using Lemma \ref{toral} it follows that
$\Gamma$ has no nontrivial fixed points in the toral algebra ${\mathfrak t}_1$.
Hence, $H^2 ({\mathcal O}_X, \r)\,=\,0$ by Theorem \ref{nilpotent}.

{\bf Proof for ${\mathfrak g} ={\mathfrak e}_8$:}\,
There are $69$ nilpotent orbits in ${\mathfrak e}_8$ (see Table 5 in \cite{A}). 
Let $X\,\in\, {\mathfrak e}_8$ be a nilpotent element and
let $(X,E,Y)$ be a ${\mathfrak s} {\mathfrak l}_2$-triple
in ${\mathfrak e}_8$ containing $X$. From Column 4 of Table 5 in \cite{A} 
we see the following: sixty-two of the sixty-nine nilpotent 
orbits ${\mathcal O}_X$ have the property that 
$Z_{{\mathfrak e}_8} (X,E,Y)$ is either trivial or one of the following Lie 
algebras:
$${\mathfrak e}_7,\, {\mathfrak b}_6,\, {\mathfrak f}_4 \oplus {\mathfrak 
a}_1,\, {\mathfrak e}_6,\, {\mathfrak c}_4, \, {\mathfrak a}_5,\,
{\mathfrak b}_3 \oplus {\mathfrak a}_1,\, {\mathfrak b}_5,\, {\mathfrak g}_2 
\oplus {\mathfrak a}_1,\, {\mathfrak g}_2 \oplus {\mathfrak g}_2,\,
{\mathfrak d}_4, \,{\mathfrak b}_2,\,{\mathfrak f}_4, 
$$
$$
{\mathfrak b}_2 \oplus {\mathfrak a}_1,\,
{\mathfrak a}_1 \oplus {\mathfrak a}_1 \oplus {\mathfrak a}_1, \,{\mathfrak 
a}_4,{\mathfrak a}_1 \oplus {\mathfrak a}_1,\, {\mathfrak a}_2,\,
{\mathfrak c}_3,{\mathfrak a}_3, \,{\mathfrak a}_1, \,{\mathfrak g}_2,\, 
{\mathfrak b}_3, \, {\mathfrak b}_2\, .
$$
As all the above Lie algebras are semisimple, these sixty-two nilpotent 
orbits ${\mathcal O}_X$ have the property that the center of the 
Lie algebra $Z_{{\mathfrak e}_8} (X,E,Y)$ is trivial. So,
$H^2 ({\mathcal O}_X, \r)=0$ in these cases (see Theorem 
\ref{nilpotent}).

We now consider the remaining $7$ nilpotent orbits ${\mathcal O}_X$ for which the centers of the 
Lie algebras $Z_{{\mathfrak e}_8} (X,E,Y)$ are nontrivial. The Lie 
algebras
$Z_{{\mathfrak e}_8} (X,E,Y)$ for these seven nilpotent orbits are as follows:
$$
{\mathfrak b}_2 \oplus {\mathfrak t}_1, \, {\mathfrak a}_2 \oplus {\mathfrak 
t}_1,{\mathfrak a}_1 \oplus {\mathfrak t}_1,\, {\mathfrak t}_1\, .
$$

We observe a few common features in the $7$ nilpotent orbits which will 
turn out to be useful. From Column 6, Rows 18, 24, 26,
46, 49, 52, 55 of Table 5 in \cite{A} it follows that the exact sequence
$$
1 \,\longrightarrow \,Z_G (X,E,Y)^0 \,\longrightarrow \,Z_G (X,E,Y) 
\, \longrightarrow\, \Gamma\, \longrightarrow \, 1
$$
splits and consequently,
$$Z_G (X,E,Y) \,\simeq \, \Gamma \ltimes Z_G (X,E,Y)^0\,. $$
In particular, $\Gamma$ can be identified with a subgroup of $Z_G (X,E,Y)$.
Moreover, from Column 5, Rows 18, 24, 26, 46, 49, 52, 55 of Table 5 in \cite{A}
we note that $\Gamma\,\simeq\, {\mathbb Z}/2{\mathbb Z}$. As in the 
previous 
situation, the generator of $\Gamma$ is denoted by $c$.

{\it The case when $Z_{{\mathfrak e}_8} (X,E,Y) = {\mathfrak b}_2 \oplus {\mathfrak t}_1$:} 
This case occurs in Row 18 of Table 5 in \cite{A}. Clearly, as 
${\mathfrak b}_2$ is the semisimple part of ${\mathfrak b}_2 \oplus 
{\mathfrak t}_1$, the group
$\Gamma$ normalizes ${\mathfrak b}_2$. In other words, ${\rm Ad}(c) 
{\mathfrak b}_2\,=\,{\mathfrak b}_2$.
Using part (2) of Lemma \ref{new-dim} we see that ${\mathfrak b}_2$ 
admits a toral subalgebra ${\mathfrak t}$,
with $\dim {\mathfrak t}\,=\, 2$, such that ${\rm Ad}(c) x \,=\, x$, for 
all $x \,\in \,{\mathfrak t}$.
But the maximum possible dimension of a toral subalgebra of ${\mathfrak b}_2 \oplus {\mathfrak t}_1$ 
centralized by $\Gamma$ is two (see Column 3, Row 18 of Table 5 in 
\cite{A}). Thus,
$\Gamma$ acts nontrivially on the one dimensional center ${\mathfrak t}_1$ of ${\mathfrak a}_3 \oplus {\mathfrak t}_1$.
Appealing to Theorem \ref{nilpotent} we conclude that $H^2 ({\mathcal 
O}_X, \, \r)\,=\,0$. 
 
{\it The case when $Z_{{\mathfrak e}_8} (X,E,Y) = {\mathfrak a}_2 \oplus 
{\mathfrak t}_1$:} 
We refer to Row 24 of Table 5 in \cite{A} for necessary results.
It follows, exactly in the same way, as in the case when $Z_{{\mathfrak 
e}_7} (X,E,Y) = {\mathfrak a}_2 \oplus {\mathfrak t}_1$,
that $\Gamma$ must act on the center ${\mathfrak t}_1$ of ${\mathfrak a}_2 \oplus {\mathfrak t}_1$ nontrivially.
We omit the details to avoid repetitive arguments. 
Hence, by Theorem \ref{nilpotent}, $H^2 ({\mathcal O}_X, \r)=0$. 

{\it The case when $Z_{{\mathfrak e}_8} (X,E,Y) = {\mathfrak a}_1 \oplus {\mathfrak t}_1$:} This case occurs in two
rows, namely in Rows 19 and 26 of Table 4 in \cite{A}.
As in the case above, it follows exactly in the same way as in the case when $Z_{{\mathfrak e}_7} (X,E,Y) = {\mathfrak a}_1 \oplus {\mathfrak t}_1$,
that $\Gamma$ must act on the center ${\mathfrak t}_1$ of ${\mathfrak a}_1 \oplus {\mathfrak t}_1$ nontrivially. 
Using Theorem \ref{nilpotent} we see that $H^2 ({\mathcal O}_X, \r)=0$. 

{\it The case when $Z_{{\mathfrak e}_8} (X,E,Y) = {\mathfrak t}_1$:} 
This appears four times in Table 4 of \cite{A}, namely in Row 46, Row 
49, Row 52 and in Row 55. We deal all the cases simultaneously.
In all the cases the possible maximal diagonalizable subgroup of $Z_G (X,E,Y)$ is isomorphic to either $T_1$
or a finite abelian group. So by Lemma \ref{toral} it follows that 
$\Gamma$ must act on $Z_{{\mathfrak e}_8} (X,E,Y)={\mathfrak t}_1$ nontrivially. Hence 
by Theorem \ref{nilpotent}, $H^2 ({\mathcal O}_X, \r)=0$. 
Thus the proof of Theorem \ref{gfe} is completed.
\end{proof}

\begin{theorem}\label{e_6} Let $G$ be a complex simple group of exceptional type
with Lie algebra ${\mathfrak e}_6$. Let $X \in {\mathfrak e}_6$ be a nilpotent element.
Then $H^2 ({\mathcal O}_X, \r)\,=\,0$ for all but nine nilpotent orbits, 
where ${\mathcal O}_X$ is the orbit of $X$ in ${\mathfrak e}_6$. For
the remaining nine nilpotent orbits,
$$\dim H^2 ({\mathcal O}_X, \r) \,=\,1\, .$$ 
\end{theorem}

\begin{proof} We may assume that $G$ is of adjoint type. From Table 3 in
\cite{A} we see that there are $20$ nilpotent
orbits in ${\mathfrak e}_6$. Let $X \in {\mathfrak e}_6$ be a nilpotent element.
Let $(X,E,Y)$ be a ${\mathfrak s} {\mathfrak l}_2$-triple
in ${\mathfrak e}_6$ containing $X$. From Column 4 of Table 3 in \cite{A} we 
see that, out of $20$ distinct conjugacy classes of nilpotent
elements, $10$ nilpotent orbits ${\mathcal O}_X$ have the property that 
$Z_{{\mathfrak e}_6} (X,E,Y)$ is either trivial or one of the following Lie algebras:
$$
{\mathfrak a}_5, \, {\mathfrak a}_2 \oplus {\mathfrak a}_1, \,{\mathfrak a}_2 
\oplus {\mathfrak a}_2,\, {\mathfrak g}_2,\, {\mathfrak a}_1,\, {\mathfrak 
a}_2\, . 
$$
As all the above Lie algebras are semisimple, $H^2 ({\mathcal O}_X, 
\r)\,=\,0$ for these ten nilpotent orbits (see Theorem
\ref{nilpotent}).

We now consider the remaining $10$ nilpotent orbits ${\mathcal O}_X$ for which the centers of the 
Lie algebras $Z_{{\mathfrak e}_6} (X,E,Y)$ are nontrivial. We list all the possible description of $Z_{{\mathfrak e}_6} (X,E,Y)$
for these $10$ nilpotent orbits ${\mathcal O}_X$: 
$$
{\mathfrak b}_3 \oplus {\mathfrak t}_1, \, {\mathfrak a}_2 \oplus {\mathfrak 
t}_1, \, {\mathfrak a}_1 \oplus {\mathfrak t}_1,\, {\mathfrak b}_2
\oplus {\mathfrak t}_1, \,{\mathfrak t}_2,\, {\mathfrak t}_1\, .
$$

{\it The cases where $Z_{{\mathfrak e}_8} (X,E,Y) \,=\, {\mathfrak b}_3 \oplus 
{\mathfrak t}_1$ or
${\mathfrak a}_2 \oplus {\mathfrak t}_1$ or ${\mathfrak a}_1 \oplus 
{\mathfrak t}_1$ or ${\mathfrak b}_2 \oplus {\mathfrak t}_1$:}
These cases occur in Rows 2, 5, 7, 8, 10 and 12 of Table 3 in \cite{A}. 
We will deal with all these simultaneously. We note from
Column 5 of Table 3 in \cite{A} that $Z_G (X,E,Y)/Z_G 
(X,E,Y)^0$ is trivial for all these cases. Using Theorem \ref{nilpotent}
we conclude that $\dim H^2 ({\mathcal O}_X, \r)\,=\,1$.

{\it The case of $Z_{{\mathfrak e}_8} (X,E,Y) \,=\, {\mathfrak t}_1$:}
These cases occur in Rows 14, 15 and 18 of Table 3 in \cite{A}. From 
Column 5 of Table 3 in \cite{A} it follows that $Z_G 
(X,E,Y)/Z_G (X,E,Y)^0$ is trivial for all these cases.
Hence, $\dim H^2 ({\mathcal O}_X, \r)\,=\,1$ by Theorem \ref{nilpotent}. 

{\it The case of $Z_{{\mathfrak e}_8} (X,E,Y) = {\mathfrak t}_2$:}
{}From Column 5, Row 11 of Table 3 in \cite{A} we see 
that $\Gamma\,\simeq
\,S_3$, the symmetric group on three symbols. Moreover, $Z_G (X,E,Y)^0\,
\simeq\, T_2$. From Column 6, Row 21 of Table 4 in \cite{A} we see that 
$$
Z_G (X,E,Y) 
\,\simeq\, \Gamma \ltimes Z_G (X,E,Y)^0
\,\simeq\, \Gamma \ltimes T_2\, .
$$
Further, by Column 3, Row 11 of Table 4 in \cite{A} we see that 
the maximal diagonalizable subgroups of $Z_G (X,E,Y) \simeq \Gamma \ltimes T_2$ are of the form either ${\mathbb Z}/ 3{\mathbb Z}$ or
${\mathbb Z}/2{\mathbb Z} \times T_1$ or $T_2$. 
We claim that the set of fixed points of $T_2$ under the action of $\Gamma \simeq S_3$ has dimension zero.
If not, then there is a one dimensional sub-torus $T_1 \subset T_2$ or all of $T_2$ remain fixed by $\Gamma$. 
Consider a copy of ${\mathbb Z}/ 3{\mathbb Z}$ in $S_3$. Then either ${\mathbb Z}/ 3{\mathbb Z} \times T_1$ or
${\mathbb Z}/ 3{\mathbb Z} \times T_2$ will be a maximal
diagonalizable subgroup of $Z_G (X,E,Y) \simeq \Gamma \ltimes T_2$. This is a contradiction.
We now appeal to Theorem \ref{nilpotent} to conclude that $\dim H^2 
({\mathcal O}_X, \r)=0$. This completes the proof of the theorem.
\end{proof}

\begin{remark}{\rm In the proof of Theorem \ref{gfe} and Theorem \ref{e_6},
we used tables in \cite{A} to determine 
the dimension of a maximal toral subalgebra in $Z_{\mathfrak g} 
(X,E,Y)$ fixed by $\Gamma$ (or by cyclic subgroups when $\Gamma$ is not
cyclic). This can also be done using \cite{So}, from which it follows 
that this dimension coincides with the corank of certain pseudo--Levi 
subalgebra associated to the cyclic subgroups of $\Gamma$.}
\end{remark}

\section*{Acknowledgements}
We are very grateful to Eric Sommers for a very helpful correspondence.
We thank the referee for simplifying the proof of Claim 1 in the proof of
Proposition \ref{iff-orbit-2}. The first-named author thanks the Institute
of Mathematical Sciences at Chennai for its hospitality.

\end{document}